\documentclass[a4paper,12pt]{article}

\usepackage{latexsym}
\usepackage{amssymb}
\usepackage{theorem}
\usepackage{amsmath}
\usepackage{amscd}
\usepackage{graphicx}
\usepackage{url}

\makeatletter
\def\BOXSYMBOL{\RIfM@\bgroup\else$\bgroup\aftergroup$\fi
  \vcenter{\hrule\hbox{\vrule height.85em\kern.6em\vrule}\hrule}\egroup}
\makeatother
\newcommand{\BOX}{%
  \ifmmode\else\leavevmode\unskip\penalty9999\hbox{}\nobreak\hfill\fi
  \quad\hbox{\BOXSYMBOL}}

\makeatletter
\newenvironment{proof}[1][\proofname]{\par
  \normalfont
  \topsep6\p@\@plus6\p@ \trivlist
  \item[\hskip\labelsep{\itshape #1}\@addpunct{\bfseries.}]\ignorespaces
}{%
  \BOX\endtrivlist
}
\newcommand{\proofname}{Proof.}
\makeatother
\usepackage[dvipdfmx,colorlinks=true]{hyperref}
\hypersetup{
setpagesize=false,
 bookmarksnumbered=true,%
 bookmarksopen=true,%
 colorlinks=true,%
 linkcolor=blue,
 citecolor=blue,
 urlcolor=blue,
}

\usepackage[usenames]{color}

\newtheorem{theorem}{Theorem}[section]
\newtheorem{corollary}[theorem]{Corollary}
\newtheorem{lemma}[theorem]{Lemma}
\newtheorem{proposition}[theorem]{Proposition}
\newtheorem{remark}[theorem]{Remark}
\newtheorem{definition}[theorem]{Definition}
\theorembodyfont{\normalfont}
\newtheorem{example}[theorem]{Example}

\numberwithin{equation}{section}

\newcommand{\pmt}[1]{{\begin{pmatrix} #1  \end{pmatrix}}}
\newcommand{\demo}{\par\noindent{\it Proof. \/}\ }
\newcommand{\enD}{\hfill $\Box$\vspace{3truemm} \par}
\newcommand{\R}{\mathbb{R}}

\newcommand{\bn}{\mbox{\boldmath $n$}}
\newcommand{\bt}{\mbox{\boldmath $t$}}
\newcommand{\bs}{\mbox{\boldmath $s$}}
\newcommand{\bb}{\mbox{\boldmath $b$}}
\newcommand{\ba}{\mbox{\boldmath $a$}}

\newcommand{\bx}{\mbox{\boldmath $x$}}

\newcommand{\be}{\mbox{\boldmath $e$}}

\newcommand{\trans}[1]{{\vphantom{#1}}^t{\!#1}}
\counterwithin{figure}{section}

\begin{document}

\title{Singularities of helicoidal surfaces of frontals in the Euclidean space}

\author{N. Nakatsuyama, K. Saji, R. Shimada, M. Takahashi}

\date{\today}

\maketitle
\begin{abstract} 
We investigate helicoidal (screw) surfaces generated not only by regular curves but also by curves with singular points. 
For curves with singular points, it is useful to use frontals in the Euclidean plane.
The helicoidal surface of a frontal can naturally be considered as a generalised framed base surface. 
Moreover, we show that it is also a framed base surface under a mild condition. 
We give basic invariants and curvatures for helicoidal surfaces of frontals by using the curvatures of Legendre curves. 
Moreover, we also give criteria for singularities of helicoidal surfaces. 
\end{abstract}

\renewcommand{\thefootnote}{\fnsymbol{footnote}}

\section{Introduction}

It is well-known that the helicoidal surfaces in the Euclidean three-dimensional space are invariant under the action of the 1-parameter group of helicoidal motions and are a generalisation of rotation surfaces. 
There are a lot of investigations of helicoidal surfaces not only as minimal surfaces but also as constant mean curvature surfaces (cf. \cite{COP, HHM, YA}). 
E. Bour proved that, for a given helicoidal surface, there exists a two-parameter family of helicoidal surfaces which are isometric to it \cite{Bour}. 
See also, \cite{HHM}.
In many areas such as robot mechanics, mechanical design and computational geometry, screw 
theory has important applications, for instance, \cite{Ball, Dimentberg, Yang}. 
Moreover, there are applications to the physics of helicoidal (screw) surfaces as an optical vortices, see for example \cite{Allen,Masuda}. 
\par 
We consider the $(x,z)$-plane into $(x,y,z)$-space and give a curve in the $(x,z)$-plane, so called the profile curve, we define the helicoidal (screw) surfaces along the $z$-direction. 
If the profile curve has a singular point, then the helicoidal surface automatically has a singular point. 
Even if the profile curve is regular, the helicoidal surface may have singular points at a point $x=0$.
For curves with singular points, it is useful to use frontals in the Euclidean plane (or, Legendre curves in the unit tangent space) \cite{Fukunaga-Takahashi2013}.
\par
In \S 2, we review the theories of Legendre curves in the unit tangent bundle over the Euclidean plane, framed surfaces in the Euclidean space and generalised framed surfaces in the Euclidean space.
The helicoidal surface of a frontal can naturally be considered as a generalised framed base surface. 
Moreover, we show that it is also a framed base surface under a mild condition. 
We give basic invariants and curvatures for helicoidal surfaces of frontals by using the curvatures of Legendre curves in \S 3. 
We introduce a slice curve of the helicoidal surface which is useful to investigate singular points of helicoidal surfaces.
\par
If $x \ne 0$, singularities of the helicoidal surface are like those of the surface of revolution of a profile curve (cf. \cite{Martins-Saji-Santos-Teramoto, Takahashi-Teramoto}). 
For singular points at a point $x=0$, singularities of a helicoidal surface are of the edge type.
This phenomenon is different from that of surfaces of revolution. 
We stick our consideration into the case $x=0$ not only for regular curves but also for curves with singular points. 
In \S 4, we give criteria for singularities of helicoidal surfaces (Theorem \ref{thm:criteria}). 
As a consequence, we show that some types of singular points of helicoidal surfaces do not appear.
In \S 5, we give examples of helicoidal surfaces of frontals with singular points. 
\par
All maps and manifolds considered here are differentiable of class $C^{\infty}$ unless stated otherwise. 

\bigskip
\noindent
{\bf Acknowledgement}. 
The second author was partially supported by JSPS KAKENHI Grant Numbers JP 22K03312 and 	22KK0034.
The fourth author was partially supported by JSPS KAKENHI Grant Number JP 24K06728.

\section{Preliminaries}
We quickly review the theories of Legendre curves in the unit tangent bundle over the Euclidean plane (cf. \cite{Fukunaga-Takahashi2013}), framed surfaces in the Euclidean space (cf. \cite{Fukunaga-Takahashi2019}) and generalised framed surfaces in the Euclidean space (cf. \cite{Takahashi-Yu}).

\subsection{Legendre curves}

Let $\gamma:I\to\R^2$ and $\nu:I\to S^1$ be smooth mappings, where $I$ is an interval of $\R$ and $S^1$ is the unit circle.
We say that $(\gamma,\nu):I \to \R^2 \times S^1$ is {\it a Legendre curve} if $(\gamma,\nu)^*\theta=0$ for all $t \in I$, where $\theta$ is a canonical contact form on the unit tangent bundle $T_1 \R^2=\R^2 \times S^1$ over $\R^2$ (cf. \cite{Arnold1, Arnold2}). 
This condition is equivalent to $\dot{\gamma}(t) \cdot \nu(t)=0$ for all $t \in I$, 
where $\dot{\gamma}(t)=(d\gamma/dt)(t)$ and $\ba\cdot\bb=a_1b_1+a_2b_2$ for any $\ba=(a_1,a_2), \bb=(b_1,b_2)\in\R^2$. 
{A point $t_0\in I$ is called a {\it singular point} of $\gamma$ if $\dot{\gamma}(t_0)=0$.} 
When a Legendre curve $(\gamma,\nu):I\to\R^2\times S^1$ gives an immersion, it is called a {\it Legendre immersion}.
We say that $\gamma:I \to \R^2$ is a {\it frontal} (respectively, a {\it front}) if there exists $\nu:I \to S^1$ such that $(\gamma,\nu)$ is a Legendre curve (respectively, a Legendre immersion). 
Examples of Legendre curves see \cite{Ishikawa,Ishikawa-book}.
We have the Frenet type formula of a frontal $\gamma$ as follows.
We put on $\mu (t)=J(\nu (t))$, where $J$ is the anticlockwise rotation by angle $\pi/2$ in $\R^2$.
Then $\{\nu(t), \mu(t) \}$ is a moving frame of the frontal $\gamma(t)$ in $\R^2$ and we have the Frenet type formula,
\begin{equation}\label{Frenet.frontal}
\left(
\begin{array}{c}
\dot{\nu}(t)\\
\dot{\mu}(t)
\end{array}
\right)
=
\left(
\begin{array}{cc}
0 & \ell(t)\\
-\ell(t) & 0
\end{array}
\right)
\left(
\begin{array}{c}
\nu(t)\\
\mu(t)
\end{array}
\right), \ 
\dot\gamma(t) = \beta(t) \mu(t),
\end{equation}
where $\ell(t)=\dot\nu(t) \cdot \mu(t)$ and $\beta(t)=\dot{\gamma}(t) \cdot \mu(t)$.
We call the pair $(\ell,\beta)$ {\it the curvature of the Legendre curve}. 
By \eqref{Frenet.frontal}, we see that $(\gamma,\nu)$ is a Legendre immersion if and only if $(\ell,\beta)\neq(0,0)$.

\begin{theorem}[Existence Theorem for Legendre curves \cite{Fukunaga-Takahashi2013}] \label{existence.Legendre}
Let $(\ell,\beta):I \to \R^2$ be a smooth mapping. 
There exists a Legendre curve $(\gamma,\nu):I \to \R^2 \times S^1$ whose curvature of the Legendre curve is $(\ell, \beta)$.
\end{theorem}
Actually, we have the following.
\begin{align*}
\gamma(t) &=\left(-\int \beta(t) \sin \left( \int \ell(t)\ dt\right) dt,\ \int \beta(t) \cos \left(\int \ell(t)\ dt\right) dt\right),\\
\nu(t) &= \left(\cos \left(\int \ell(t) \ dt\right), \ \sin \left(\int \ell(t)\ dt\right) \right).
\end{align*}
\begin{definition}\label{congruent}{\rm
Let $(\gamma,\nu)$ and $(\widetilde{\gamma},\widetilde{\nu}):I \to \R^2 \times S^1$ be Legendre curves.
We say that $(\gamma,\nu)$ and $(\widetilde{\gamma},\widetilde{\nu})$ are {\it congruent as Legendre curves} if there exist a constant rotation $A \in SO(2)$ and a translation $\ba$ on $\R^2$ such that $\widetilde{\gamma}(t)=A(\gamma(t))+\ba$ and $\widetilde{\nu}(t)=A (\nu(t))$ for all $t \in I$. 
}
\end{definition}
\begin{theorem}[Uniqueness Theorem for Legendre curves \cite{Fukunaga-Takahashi2013}] \label{uniqueness.Legendre}
Let $(\gamma,\nu)$ and $(\widetilde{\gamma},\widetilde{\nu}):I \to \R^2 \times S^1$ be Legendre curves with the curvatures of Legendre curves $(\ell,\beta)$ and $(\widetilde{\ell},\widetilde{\beta})$, respectively.
Then $(\gamma,\nu)$ and $(\widetilde{\gamma},\widetilde{\nu})$ are congruent as Legendre curves if and only if $(\ell,\beta)$ and $(\widetilde{\ell},\widetilde{\beta})$ coincide.
\end{theorem}



\subsection{Framed surfaces}

Let $\R^3$ be the $3$-dimensional Euclidean space equipped with the inner product $\ba \cdot \bb = a_1 b_1 + a_2 b_2 + a_3 b_3$, 
where $\ba = (a_1, a_2, a_3)$ and $\bb = (b_1, b_2, b_3) \in \R^3$. 
The norm of $\ba$ is given by $\vert \ba \vert = \sqrt{\ba \cdot \ba}$. 
We also define the vector product
$$
\ba \times \bb=\det \left(
\begin{array}{ccc}
\be_1 & \be_2 & \be_3 \\
a_1 & a_2 & a_3 \\
b_1 & b_2 & b_3 
\end{array}
\right),
$$
where $\{ \be_1, \be_2, \be_3 \}$ is the canonical basis of $\R^3$. 
Let $U$ be a simply connected domain of $\R^2$ and $S^2$ be the unit sphere in $\R^3$, that is, $S^2=\{\ba \in \R^3| |\ba|=1\}$.
We denote a $3$-dimensional smooth manifold $\{(\ba,\bb) \in S^2 \times S^2| \ba \cdot \bb=0\}$ by $\Delta$.
\par
We say that $(\bx,\bn,\bs):U \to \R^3 \times \Delta$ is a {\it framed surface} if $\bx_u (u,v) \cdot \bn (u,v)=0$ and $\bx_v(u,v) \cdot \bn(u,v)=0$ for all $(u,v) \in U$, where $\bx_u(u,v)=(\partial \bx/\partial u)(u,v)$ and $\bx_v(u,v)=(\partial \bx/\partial v)(u,v)$. 
We say that $\bx:U \to \R^3$ is a {\it framed base surface} if there exists $(\bn,\bs):U \to \Delta$ such that $(\bx,\bn,\bs)$ is a framed surface.
\par
Similarly to the case of Legendre curves, 
the pair $(\bx,\bn):U\to\R^3\times S^2$ is said to be a {\it Legendre surface} 
if $\bx_u (u,v) \cdot \bn (u,v)=0$ and $\bx_v(u,v) \cdot \bn(u,v)=0$ for all $(u,v) \in U$. 
Moreover, when a Legendre surface $(\bx,\bn):U\to\R^3\times S^2$ gives an immersion, this is called a {\it Legendre immersion}. 
We say that $\bx:U\to\R^3$ be a {\it frontal} (respectively, a {\it front}) if there exists a map $\bn:U\to\ S^2$ such that 
the pair $(\bx,\bn):U\to\R^3\times S^2$ is a Legendre surface (respectively, a Legendre immersion).
By definition, the framed base surface is a frontal. 
At least locally, the frontal is a framed base surface. 
{For a framed surface $(\bx,\bn,\bs)$, we say that a point $p\in U$ is a {\it singular point of $\bx$} 
if $\bx$ is not an immersion at $p$.}

We denote $\bt(u,v)=\bn(u,v) \times \bs(u,v)$. 
Then $\{\bn(u,v),\bs(u,v),\bt(u,v)\}$ is a moving frame along $\bx(u,v)$.
Thus, we have the following systems of differential equations:
\begin{equation}\label{tangent}
\left(\begin{array}{c}
\bx_u \\
\bx_v
\end{array}\right)
=
\left(\begin{array}{cc} 
a_1 & b_1 \\
a_2 & b_2 
\end{array}\right)
\left(\begin{array}{c}
\bs \\
\bt
\end{array}\right),
\end{equation}
\begin{equation}\label{frame}
\left(\begin{array}{c} 
\bn_u \\
\bs_u \\
\bt_u
\end{array}\right)
=
\left(\begin{array}{ccc} 
0 & e_1 & f_1 \\
-e_1 & 0 & g_1 \\
-f_1 & -g_1 & 0
\end{array}\right)
\left(\begin{array}{c} 
\bn \\
\bs \\
\bt
\end{array}\right)
, \ 
\left(\begin{array}{c} 
\bn_v \\
\bs_v \\
\bt_v
\end{array}\right)
=
\left(\begin{array}{ccc} 
0 & e_2 & f_2 \\
-e_2 & 0 & g_2 \\
-f_2 & -g_2 & 0
\end{array}\right)
\left(\begin{array}{c} 
\bn \\
\bs \\
\bt
\end{array}\right),
\end{equation}
where $a_i,b_i,e_i,f_i,g_i:U \to \R, i=1,2$ are smooth functions and we call the functions {\it basic invariants} of the framed surface. 
We denote the matrices $(\ref{tangent})$ and $(\ref{frame})$ by $\mathcal{G}, \mathcal{F}_1, \mathcal{F}_2$, respectively. 
We also call the matrices $(\mathcal{G}, \mathcal{F}_1, \mathcal{F}_2)$  {\it basic invariants} of the framed surface $(\bx,\bn,\bs)$.
Since the integrability condition $\bx_{uv}=\bx_{vu}$ and $\mathcal{F}_{2,u}-\mathcal{F}_{1,v}
=\mathcal{F}_1\mathcal{F}_2-\mathcal{F}_2\mathcal{F}_1$, 
the basic invariants should satisfy the following conditions:
\begin{equation}\label{integrability.condition}
\begin{cases}
a_{1,v}-b_1g_2 = a_{2,u}-b_2g_1, \\
b_{1,v}-a_2g_1 = b_{2,u}-a_1g_2, \\
a_1 e_2 + b_1 f_2 = a_2 e_1 + b_2 f_1,
\end{cases}
\begin{cases}
e_{1,v}-f_1g_2 = e_{2,u}-f_2g_1, \\
f_{1,v}-e_2g_1 = f_{2,u}-e_1g_2, \\
g_{1,v}-e_1f_2 = g_{2,u}-e_2f_1. 
\end{cases}
\end{equation}
We give fundamental theorems for framed surfaces, that is, 
the existence and uniqueness theorem of framed surfaces for basic invariants.

\begin{theorem}[Existence Theorem for framed surfaces \cite{Fukunaga-Takahashi2019}]\label{existence.framed.surface}
Let $U$ be a simply connected domain in $\R^2$ and let $a_i,b_i,e_i,f_i,g_i:U \to \R, i=1,2$ be smooth functions with the integrability conditions $(\ref{integrability.condition})$.
Then there exists a framed surface $(\bx,\bn,\bs):U \to \R^3 \times \Delta$ whose associated basic invariants is $a_i,b_i,e_i,f_i,g_i, i=1,2$.
\end{theorem}
\begin{definition}\label{congruent.framed.surface}{\rm
Let $(\bx,\bn,\bs), (\widetilde{\bx},\widetilde{\bn},\widetilde{\bs}):U \to \R^3 \times \Delta$ be framed surfaces.
We say that $(\bx,\bn,\bs)$ and $(\widetilde{\bx},\widetilde{\bn},\widetilde{\bs})$ are {\it congruent as framed surfaces} if there exist a constant rotation $A \in SO(3)$ and a translation $\ba \in \R^3$ such that 
$$
\widetilde{\bx}(u,v)=A(\bx(u,v))+\ba, \ \widetilde{\bn}(u,v)=A(\bn(u,v)), \  \widetilde{\bs}(u,v)=A(\bs(u,v))
$$ 
for all $(u,v) \in U$.}
\end{definition}
\begin{theorem}[Uniqueness Theorem for framed surfaces \cite{Fukunaga-Takahashi2019}]\label{uniqueness.framed.surface}
Let $(\bx,\bn,\bs)$, $(\widetilde{\bx},\widetilde{\bn},\widetilde{\bs}):U \to \R^3 \times \Delta$ be framed surfaces 
with basic invariants $(\mathcal{G},\mathcal{F}_1,\mathcal{F}_2)$ and $(\widetilde{\mathcal{G}},\widetilde{\mathcal{F}}_1,\widetilde{\mathcal{F}}_2)$, respectively.
Then $(\bx,\bn,\bs)$ and $(\widetilde{\bx},\widetilde{\bn},\widetilde{\bs})$ are congruent as framed surfaces if and only if $(\mathcal{G},\mathcal{F}_1,\mathcal{F}_2)$ and $(\widetilde{\mathcal{G}},\widetilde{\mathcal{F}}_1,\widetilde{\mathcal{F}}_2)$  coincide.
\end{theorem}

Let $(\bx,\bn,\bs):U \to \R^3 \times \Delta$ be a framed surface with basic invariants $(\mathcal{G},\mathcal{F}_1,\mathcal{F}_2).$

\begin{definition}\label{curvature.framed.surface}{\rm 
We define a smooth mapping $C^F=(J^F,K^F,H^F):U \to \R^3$ by 
\begin{align*}
J^F =
\det \left(\begin{array}{cc} 
a_1 & b_1 \\ 
a_2 & b_2
\end{array}\right), 
K^F =
\det \left(\begin{array}{cc} 
e_1 & f_1 \\ 
e_2 & f_2
\end{array}\right), 
H^F = -\frac{1}{2}\left\{ 
\det \left(\begin{array}{cc} 
a_1 & f_1 \\ 
a_2 & f_2
\end{array}\right)
-
\det \left(\begin{array}{cc} 
b_1 & e_1 \\ 
b_2 & e_2
\end{array}\right)
\right\}.
\end{align*}
We call $C^F = (J^F,K^F,H^F)$ a {\it curvature of the framed surface}.
}
\end{definition}
\begin{remark}{\rm 
If the surface $\bx$ is regular, then we have $K=K^F/J^F$ and $H=H^F/J^F$, where 
$K$ is the Gauss curvature and $H$ is the mean curvature of the regular surface (cf. \cite{Fukunaga-Takahashi2019}). 
For relations between behaviour of the Gauss curvature, the mean curvature of fronts at non-degenerate singular points and geometric invariants of fronts see \cite{Martins-Saji-Umehara-Yamada}.} 
\end{remark}
We say that $(\bx,\bn,\bs):U \to \R^3 \times \Delta$ is a {\it framed immersion} if $(\bx,\bn,\bs)$ is an immersion.
\begin{proposition}[\cite{Fukunaga-Takahashi2019}]
\label{immersive.condition}
Let $(\bx,\bn,\bs):U \to \R^3 \times \Delta$ be a framed surface and $p \in U$.
\par
$(1)$ $\bx$ is an immersion (a regular surface) around $p$ if and only if $J^F (p) \not=0$.
\par
$(2)$ $(\bx,\bn)$ is a Legendre immersion around $p$ if and only if $C^F (p) \not=0$.
\end{proposition}




\subsection{Generalised framed surfaces}

Let $(\bx,\nu_1,\nu_2):U \to \R^3 \times \Delta$ be a smooth mapping.
We denote $\nu=\bx_u \times \bx_v$.

\begin{definition}\label{generalised.framed.surface}{\rm
We say that $(\bx,\nu_1,\nu_2):U \to \R^3 \times \Delta$ is a {\it generalised framed surface} if there exist smooth functions $\alpha, \beta:U \to \R$ such that $\nu(u,v)=\alpha(u,v) \nu_1(u,v)+\beta(u,v) \nu_2(u,v)$ for all $(u,v) \in U$.
We say that $\bx:U \to \R^3$ is a {\it generalised framed base surface} if there exists $(\nu_1,\nu_2):U \to \Delta$ such that $(\bx,\nu_1,\nu_2)$ is a generalised framed surface.
}
\end{definition}

We denote $\nu_3(u,v)=\nu_1(u,v) \times \nu_2(u,v)$.
Then $\{\nu_1(u,v),\nu_2(u,v),\nu_3(u,v)\}$ is a moving frame along $\bx(u,v)$
and we have the following systems of differential equations:
$$
\begin{pmatrix}
\bx_u \\
\bx_v
\end{pmatrix}
=
\begin{pmatrix}
a_1 & b_1 & c_1\\
a_2 & b_2 & c_2
\end{pmatrix}
\begin{pmatrix}
\nu_1 \\
\nu_2 \\
\nu_3
\end{pmatrix},
\begin{pmatrix}
\nu_{1u} \\
\nu_{2u} \\
\nu_{3u}
\end{pmatrix}
=
\begin{pmatrix}
0 & e_1 & f_1 \\
-e_1 & 0 & g_1 \\
-f_1 & -g_1 & 0
\end{pmatrix}
\begin{pmatrix}
\nu_1 \\
\nu_2 \\
\nu_3
\end{pmatrix}, 
$$
$$
\begin{pmatrix}
\nu_{1v} \\
\nu_{2v} \\
\nu_{3v}
\end{pmatrix}
=
\begin{pmatrix}
0 & e_2 & f_2 \\
-e_2 & 0 & g_2 \\
-f_2 & -g_2 & 0
\end{pmatrix}
\begin{pmatrix}
\nu_1 \\
\nu_2 \\
\nu_3
\end{pmatrix},
$$
where $a_i,b_i,c_i,e_i,f_i,g_i:U \to \R, i=1,2$ are smooth functions with $a_1b_2-a_2b_1=0$.
We call the functions {\it basic invariants} of the generalised framed surface.
We denote the above matrices by $\mathcal{G}, \mathcal{F}_1, \mathcal{F}_2$, respectively.
We also call the matrices $(\mathcal{G}, \mathcal{F}_1, \mathcal{F}_2)$ {\it basic invariants} of the generalised framed surface $(\bx,\nu_1,\nu_2)$.
By definition, we have
$$
\alpha(u,v)={\det}
\begin{pmatrix}
b_1(u,v) & c_1(u,v)\\
b_2(u,v) & c_2(u,v)
\end{pmatrix}, \
\beta(u,v)=-{\det}
\begin{pmatrix}
a_1(u,v) & c_1(u,v)\\
a_2(u,v) & c_2(u,v)
\end{pmatrix}.
$$

Since the integrability conditions $\bx_{uv}=\bx_{vu}$ and $\mathcal{F}_{2u}-\mathcal{F}_{1v}
=\mathcal{F}_1\mathcal{F}_2-\mathcal{F}_2\mathcal{F}_1$,
the basic invariants should be satisfied the following conditions:
\begin{align}
&
\begin{cases}\label{integrability.condition.tangent-GFS}
a_{1v}-b_1e_2-c_1f_2 = a_{2u}-b_2e_1-c_2f_1, \\
b_{1v}+a_1e_2-c_1g_2 = b_{2u}+a_2e_1-c_2g_1, \\
c_{1v}+a_1 f_2 + b_1 g_2 =c_{2u}+ a_2 f_1 + b_2 g_1,
\end{cases}
\\
&
\begin{cases}\label{integrability.condition.normal-GFS}
e_{1v}-f_1g_2 = e_{2u}-f_2g_1, \\
f_{1v}-e_2g_1 = f_{2u}-e_1g_2, \\
g_{1v}-e_1f_2 = g_{2u}-e_2f_1.
\end{cases}
\end{align}

We give fundamental theorems for generalised framed surfaces, that is,
the existence and uniqueness theorems for the basic invariants of generalised framed surfaces.

\begin{theorem}[Existence Theorem for generalised framed surfaces \cite{Takahashi-Yu}]\label{existence.GFS}
Let $U$ be a simply connected domain in $\R^2$ and let $(a_i,b_i,c_i,e_i,f_i,g_i):U \to \mathbb{R}^{12}, i=1,2$ be a smooth mapping satisfying $a_1b_2-a_2b_1=0$, the integrability conditions \eqref{integrability.condition.tangent-GFS} and \eqref{integrability.condition.normal-GFS}.
Then there exists a generalised framed surface $(\bx,\nu_1,\nu_2):U \to \R^3 \times \Delta$ whose associated basic invariants are $(\mathcal{G},\mathcal{F}_1,\mathcal{F}_2)$.
\end{theorem}
\begin{definition}
{\rm
Let $(\bx,\nu_1,\nu_2), (\widetilde{\bx},\widetilde{\nu}_1,\widetilde{\nu}_2):U \to \mathbb{R}^3 \times \Delta$ be generalised framed surfaces.
We say that $(\bx,\nu_1,\nu_2)$ and $(\widetilde{\bx},\widetilde{\nu}_1,\widetilde{\nu}_2)$ are {\it congruent as generalised framed surfaces} if there exist a constant rotation $A \in SO(3)$ and a translation $\ba \in \mathbb{R}^3$ such that $\widetilde{\bx}(u,v) = A(\bx(u,v)) +\ba$, $\widetilde{\nu}_1(u,v) = A(\nu_1(u,v))$ and $\widetilde{\nu}_2(u,v) = A(\nu_2(u,v))$ for all $(u,v) \in U$.
}
\end{definition}
\begin{theorem}[Uniqueness Theorem for generalised framed surfaces \cite{Takahashi-Yu}]\label{uniqueness.GFS}
Let $(\bx,\nu_1,\nu_2),\\ (\widetilde{\bx},\widetilde{\nu}_1,\widetilde{\nu}_2):U \to \R^3 \times \Delta$ be generalised framed surfaces with basic invariants  $(\mathcal{G},\mathcal{F}_1,\mathcal{F}_2), (\widetilde{\mathcal{G}},\widetilde{\mathcal{F}}_1,\\ \widetilde{\mathcal{F}}_2)$, respectively.
Then $(\bx,\nu_1,\nu_2)$ and $(\widetilde{\bx},\widetilde{\nu}_1,\widetilde{\nu}_2)$ are congruent as generalised framed surfaces if and only if the basic invariants $(\mathcal{G},\mathcal{F}_1,\mathcal{F}_2)$ and $(\widetilde{\mathcal{G}},\widetilde{\mathcal{F}}_1,\widetilde{\mathcal{F}}_2)$ coincide.
\end{theorem}
Note that a framed surface is a generalised framed surface.
Conversely, we have the following.
\begin{theorem}[\cite{Takahashi-Yu}]\label{FBS.condition}
Let $(\bx,\nu_1,\nu_2):U \to \R^3 \times \Delta$ be a generalised framed surface with $\nu=\alpha\nu_1+\beta\nu_2$.
\par
$(1)$ If $\bx$ is a framed base surface, then the functions $\alpha$ and $\beta$ are linearly  dependent.
\par
$(2)$ Suppose that the set of regular points of $\bx$ is dense in $U$. 
If  the functions $\alpha$ and $\beta$ are linearly  dependent, then $\bx$ is a framed base surface at least locally.
\end{theorem}

\section{Helicoidal surfaces of frontals}

Let $(\gamma,\nu):I \to \R^2 \times S^1$ be  a Legendre curve with the curvature $(\ell,\beta)$. 
We denote $\gamma(u)=(x(u),z(u))$ and $\nu(u)=(a(u),b(u))$. 
By definition and the Frenet type formula (\ref{Frenet.frontal}), we have $\dot{x}(u)a(u)+\dot{z}(u)b(u)=0$, $a^2(u)+b^2(u)=1$, 
\begin{equation}\label{Frenet-type}
\left(\begin{array}{c} 
\dot{x}(u) \\
\dot{z}(u)
\end{array}\right)=
\beta(u)
\left(\begin{array}{c} 
-b(u) \\
a(u)
\end{array}\right), 
\left(\begin{array}{c} 
\dot{a}(u) \\
\dot{b}(u)
\end{array}\right)=
\ell(u)
\left(\begin{array}{c} 
-b(u) \\
a(u)
\end{array}\right)
\end{equation}
for all $u \in I$. 
We consider $\R^2 \subset \R^3$ as $(x,z)$-plane into $(x,y,z)$-space. 
We call $\gamma$ {\it a profile curve} (cf. \cite{Gray}).

\begin{definition}{\rm
For the frontal $\gamma:I \to \R^2, \gamma(u)=(x(u),z(u))$, we say that 
\begin{align}\label{helicoidal-surface}
r:I \times \R \to \R^3, r(u,v)=(x(u)\cos v,x(u)\sin v,z(u)+\lambda v)
\end{align}
is a {\it helicoidal surface} (or,  {\it screw surface}), where $\lambda$ is a non-zero constant.
}
\end{definition}

By a direct calculation, we have 
\begin{align*}
r_u(u,v) &=\beta(u)(-b(u)\cos v,-b(u)\sin v,a(u)), \\
r_v(u,v) &=(-x(u)\sin v,x(u)\cos v,\lambda).
\end{align*}
Since
\begin{align*}
\nu(u,v) &=r_u(u,v) \times r_v(u,v)\\
&=-\beta(u)x(u)(a(u)\cos v,a(u)\sin v,b(u))+\beta(u)b(u)\lambda(-\sin v,\cos v,0),
\end{align*}
$(u,v)$ is a singular point of $r$ if and only if $\beta(u)=0$ or $(b(u),x(u))=(0,0)$. 
Therefore, we have three cases: 
\begin{align*}
(1) & \ \beta(u)=0, \ (x(u),b(u)) \not=(0,0), \\
(2) & \ \beta(u)\not=0, \ (x(u),b(u))=(0,0), \\
(3) & \ \beta(u)=0, \ (x(u),b(u))=(0,0).
\end{align*}
Since $x(u)=0$ is a special case of the singular point of $r$, we consider this case in section \ref{S4}.
\par
We denote $\nu_1(u,v)=(a(u)\cos v,a(u)\sin v,b(u))$ and $\nu_2(u,v)=(-\sin v,\cos v,0)$.
\begin{proposition}
Under the above notations, 
$(r,\nu_1,\nu_2):I \times \R \to \R^3 \times \Delta$ is a generalised framed surface with 
$\alpha_r(u,v)=-\beta(u)x(u), \beta_r(u,v)=\beta(u)b(u) \lambda$ and the basic invariants
\begin{align*}
\begin{pmatrix}
a_1(u,v)& b_1(u,v) & c_1(u,v)\\
a_2(u,v) & b_2(u,v) & c_2(u,v)
\end{pmatrix}
&= \begin{pmatrix}
0 & 0 & \beta(u) \\
\lambda b(u) & x(u) & \lambda a(u)
\end{pmatrix}
,\\
\begin{pmatrix}
e_1(u,v) & f_1(u,v) & g_1(u,v)\\
e_2(u,v) & f_2(u,v) & g_2(u,v)
\end{pmatrix}
&=  \begin{pmatrix}
0 & \ell(u) & 0 \\
a(u) & 0 & b(u)
\end{pmatrix}.
\end{align*}
\end{proposition}
\demo
By a direct calculation, we have $r_u(u,v) =\beta(u)(-b(u)\cos v,-b(u)\sin v,a(u))$ and $r_v(u,v) =(-x(u)\sin v,x(u)\cos v,\lambda)$. Since
\begin{align*}
\nu_3(u,v) &=\nu_1(u,v) \times\nu_2(u,v)=(-b(u)\cos v,-b(u)\sin v,a(u)),
\end{align*}
and $r_u(u,v)=\beta(u)\nu_3(u,v)$, we have $a_1(u,v)=0, b_1(u,v)=0$ and $c_1(u,v)=\beta(u)$. 
Moreover, since 
$$
r_v(u,v)=a_2(u,v)\nu_1(u,v)+b_2(u,v)\nu_2(u,v)+c_2(u,v)\nu_3(u,v),
$$ 
we have $a_2(u,v)=\lambda b(u), b_2(u,v)=x(u)$ and $c_2(u,v)=\lambda a(u)$. 
By a direct calculation, we have
\begin{align*}
\nu_{1u}(u,v)&=\ell(u)(-b(u)\cos v, -b(u)\sin v, a(u)),\\
\nu_{2u}(u,v)&=(0, 0, 0),\\
\nu_{3u}(u,v)&=-\ell(u)(a(u)\cos v, a(u)\sin v, b(u)),\\
\nu_{1v}(u,v)&=(-a(u)\sin v, a(u)\cos v, 0),\\
\nu_{2v}(u,v)&=(-\cos v, -\sin v, 0),\\
\nu_{3v}(u,v)&=(b(u)\sin v, -b(u)\cos v, 0).
\end{align*}
Therefore, we have $e_1(u,v)=0, f_1(u,v)=\ell(u), g_1(u,v)=0, e_2(u,v)=a(u), f_2(u,v)=0$ and $g_2(u,v)=b(u)$. 
\enD
Suppose that the set of regular points of $r$ is dense in $I \times \R$.
Since $\alpha_r$ and $\beta_r$ does not depend on the parameter $v$, we 
denote $\alpha_r(u,v)$ and $\beta_r(u,v)$ by $\alpha_r(u)$ and $\beta_r(u)$, respectively. 
If $\alpha_r$ and $\beta_r$ are linearly independent, that is, there exists a non-zero smooth mapping $(k_1,k_2):I \to \R^2 \setminus \{0\}$ such that 
\begin{align}\label{linearly-dependent}
k_1(u) \alpha_r(u)+k_2(u)\beta_r(u)=0
\end{align} 
for all $u \in I$, then the helicoidal surface $r$ is a framed base surface at least locally by Theorem \ref{FBS.condition}. 
Moreover, by the assumption, the set of regular points of the frontal $\gamma$ is also dense in $I$, that is, the set of points of $\beta(t) \not=0$ is dense in $I$, then condition \eqref{linearly-dependent} is given by  
$-k_1(u) x(u)+k_2(u)b(u)\lambda=0$ for all $u \in I$. 
Since $(k_1(u),k_2(u)) \not=(0,0)$, we may assume that $k_1^2(u)+k^2_2(u)=1$ for all $u \in I$. 
Then we denote 
\begin{align*}
\bn(u,v)&=(k_2(u)a(u)\cos v+k_1(u)\sin v, k_2(u)a(u)\sin v-k_1(u)\cos v,k_2(u)b(u)), \\
\bs(u,v)&=(-b(u)\cos v,-b(u)\sin v,a(u)).
\end{align*}

\begin{proposition} \label{basic-invariants-framed-surface}
Suppose that there exists $(k_1,k_2):I \to \R^2$ with $k_1^2(u)+k^2_2(u)=1$ such that $-k_1(u) x(u)+k_2(u)b(u)\lambda=0$ for all $u \in I$. 
Then $(r,\bn,\bs):I \times \R \to \R^3 \times \Delta$ is a framed surface with the basic invariants 
\begin{align*}
\begin{pmatrix}
a_1(u,v)& b_1(u,v) \\
a_2(u,v) & b_2(u,v)
\end{pmatrix}
&= \begin{pmatrix}
\beta(u) & 0 \\
\lambda a(u) & -k_2(u)x(u)-k_1(u)b(u)\lambda
\end{pmatrix}
,\\
\begin{pmatrix}
e_1(u,v) & f_1(u,v) & g_1(u,v)\\
e_2(u,v) & f_2(u,v) & g_2(u,v)
\end{pmatrix}
&=  \begin{pmatrix}
k_2(u)\ell(u) & k_2(u)\dot{k}_1(u)-k_1(u)\dot{k}_2(u) & k_1(u)\ell(u) \\
-k_1(u)b(u) & -a(u) & k_2(u)b(u)
\end{pmatrix}.
\end{align*}
\end{proposition}
\demo
By a direct calculation, we have $r_u(u,v) =\beta(u)(-b(u)\cos v,-b(u)\sin v,a(u))$ and $r_v(u,v) =(-x(u)\sin v,x(u)\cos v,\lambda)$. Since 
\begin{align*}
\bt(u,v)&=\bn(u,v)\times\bs(u,v)\\
&=(k_2(u)\sin v+k_1(u)a(u)\cos v, -k_2(u)\cos v-k_1(u)a(u)\sin v,-k_1(u)b(u)),
\end{align*}
and $r_u(u,v)=\beta(u)\bs(u,v)$, we have $a_1(u,v)=\beta(u)$ and $b_1(u,v)=0$. Moreover, since $r_v(u,v)=a_2(u,v)\bs(u,v)+b_2(u,v)\bt(u,v)$, we have $a_2(u,v)=\lambda a(u)$ and $b_2(u,v)=-k_2(u)x(u)-k_1(u)b(u)\lambda$. 
By a direct calculation, we have
\begin{align*}
\bn_{u}(u,v)&=
((\dot{k}_2(u)a(u)-k_2(u)\ell(u)b(u))\cos v+\dot{k}_1(u)\sin v, \\
&\qquad (\dot{k}_2(u)a(u)-k_2(u)\ell(u)b(u))\sin v-\dot{k}_1(u)\cos v, \\
&\qquad \dot{k}_2(u)b(u)+k_2(u)\ell(u){a}(u)),\\
\bs_{u}(u,v)&=(\ell(u){a}(u)\cos v, \ell(u){a}(u)\sin v, -\ell(u){b}(u)),\\
\bn_{v}(u,v)&=(-k_2(u)a(u)\sin v+k_1(u)\cos v, k_(u)a(u)\cos v+k_1(u)\sin v, 0),\\
\bs_{v}(u,v)&=(b(u)\sin v, -b(u)\cos v, 0).
\end{align*}
Therefore, we have 
$e_1(u,v)=k_2(u)\ell(u), f_1(u,v)=k_2(u)\dot{k}_1(u)-k_1(u)\dot{k}_2(u), g_1(u,v)=k_1(u)\ell(u),\\
e_2(u,v)=k_1(u)b(u), f_2(u,v)=-a(u)$ and $g_2(u,v)=k_2(u)b(u).$
\enD
By a direct calculation, we have the curvature of the helicoidal surface of the frontal.
\begin{corollary}\label{curvature}
By the same assumption of Proposition \ref{basic-invariants-framed-surface}, we have the curvature of the helicoidal surface of the frontal as follows.
\begin{align*}
J^F(u,v) &=-\beta(u)(k_2(u)x(u)+k_1(u)b(u)\lambda),\\
K^F(u,v) &= (k_2(u)b(u))_u=\dot{k}_2(u)b(u)+k_2(u)\ell(u) a(u),\\
H^F(u,v) &= -\frac{1}{2}\Bigl(-\beta(u)a(u)-\lambda a(u)(k_2(u)\dot{k}_1(u)-k_1(u)\dot{k}_2(u))\\
&\quad -(k_2(u)x(u)+k_1(u)b(u)\lambda)k_2(u)\ell(u) \Bigr).
\end{align*}
\end{corollary}
\par
In particular, if $(x(u),b(u)) \not=(0,0)$ for all $u \in I$, then we can always take 
$$
k_1(u)=\frac{b(u)\lambda}{\sqrt{x^2(u)+b^2(u)\lambda^2}}, \ k_2(u)=\frac{x(u)}{\sqrt{x^2(u)+b^2(u)\lambda^2}}.
$$
Therefore, in this case, the helicoidal surface $r$ is a framed base surface (or, a frontal).
\par
The next we consider a geometric property of helicoidal surfaces. 
We investigate parallel surfaces of the helicoidal surface. 
We use the following notations. 
\par
Let $(\gamma,\nu):I \to \R^2 \times S^1$ be a Legendre curve with \eqref{Frenet-type}. 
We denote $r=r[\gamma]: I \times \R \to \R^3$ is a helicoidal surface, that is, 
$r[\gamma](u,v)=(x(u)\cos v,x(u)\sin v,z(u)+\lambda v)$, where $\lambda$ is a non-zero constant  \eqref{helicoidal-surface}. 
We define a smooth curve $s=s[r[\gamma]]: I \to \R^2$ by 
$$
s[r[\gamma]](u)=\left( x(u)\cos \frac{z(u)}{\lambda}, -x(u)\sin \frac{z(u)}{\lambda} \right)
$$ 
and call it a {\it slice curve} of the helicoidal surface $r[\gamma]$. 
The slice curve cuts the helicoidal surface of the $xy$-plane and it is useful to investigate singular points of helicoidal surfaces in \S 4. 

We consider relation between parallel surfaces of the helicoidal surface and parallel curves of slice curves. 
 
Suppose that there exists $(k_1,k_2):I \to \R^2$ with $k_1^2(u)+k^2_2(u)=1$ such that $-k_1(u) x(u)+k_2(u)b(u)\lambda=0$ for all $u \in I$. 
Then $(r,\bn,\bs):I \times \R \to \R^3 \times \Delta$ is a framed surface by Proposition \ref{basic-invariants-framed-surface}. 
The parallel surface of the helicoidal surface is given by 
\begin{align*}
r[\gamma]^{\widetilde{t}}(u,v)&=r[\gamma](u,v)+\widetilde{t}\bn(u,v) \\
&=((x(u)+\widetilde{t}k_2(u)a(u))\cos v+\widetilde{t}k_1(u)\sin v, \\
&\qquad (x(u)+\widetilde{t}k_2(u)a(u))\sin v-\widetilde{t}k_1(u)\cos v, \\
&\qquad z(u)+\lambda v+\widetilde{t}k_2(u)b(u))
\end{align*}
for $\widetilde{t} \in \R \setminus \{0\}$. 
Suppose that there exist smooth functions $A, \theta:I \to \R$ such that 
$$
x(u)+\widetilde{t}k_2(u)a(u)=A(u)\cos \theta(u), \ \widetilde{t}k_1(u)=A(u)\sin \theta(u).
$$
Then 
$$
r[\gamma]^{\widetilde{t}}(u,v)=(A(u)\cos(v-\theta(u)),A(u)\sin(v-\theta(u)),z(u)+\lambda(v-\theta(u))+\lambda \theta(u)+\widetilde{t}k_2(u)b(u)).
$$
By the parameter change $v-\theta(u)=\widetilde{v}$, we have 
$$
r[\gamma]^{\widetilde{t}}(u,\widetilde{v})=(A(u)\cos \widetilde{v},A(u)\sin \widetilde{v},z(u)+\lambda \widetilde{v}+\lambda \theta(u)+\widetilde{t}k_2(u)b(u)).
$$
It follows that the slice curve of the parallel surface of helicoidal surface is given by
\begin{align}\label{slice-parallel}
s[r[\gamma]^{\widetilde{t}}](u)=\left(A(u) \cos \frac{z(u)+\lambda \theta(u)+\widetilde{t}k_2(u)b(u)}{\lambda},  -A(u) \sin \frac{z(u)+\lambda \theta(u)+\widetilde{t}k_2(u)b(u)}{\lambda} \right).
\end{align}

We give a condition that the slice curve of helicoidal surface is a frontal.

\begin{proposition}\label{slice-curve-frontal}
Suppose that there exists $(\ell_1,\ell_2):I \to \R^2$ with $\ell_1^2(u)+\ell^2_2(u)=1$ such that $\ell_1(u) b(u)\lambda +\ell_2(u)x(u)a(u)=0$ for all $u \in I$. 
Then $(s[r(\gamma)],\nu^s):I \to \R^2 \times S^1$ is a Legendre curve with the curvature 
\begin{align*} 
\ell^s(u) &=-a(u)\beta(u)+\ell_1(u)\dot{\ell}_2(u)-\dot{\ell}_1(u)\ell_2(u),\\
\beta^s(u) &=(\ell_1(u)x(u)a(u)-\ell_2(u)b(u))\beta(u).
\end{align*}
\end{proposition}
\demo
Differentiating $s[r[\gamma]](u)$ and noticing \eqref{Frenet-type}, we have 
\begin{align*}
\frac{d}{du} s[r[\gamma]](u)
&=\biggl(-\beta(u)b(u)\cos\frac{z(u)}{\lambda}-x(u)\beta(u)a(u)\sin\frac{z(u)}{\lambda},\\
&\hspace{20mm}\beta(u)b(u)\sin\frac{z(u)}{\lambda}-x(u)\beta(u)a(u)\cos\frac{z(u)}{\lambda} \biggr).
\end{align*}
Moreover, by differentiating 
$$
\nu^s(u)=\left(-\ell_2(u)\sin\frac{z(u)}{\lambda}-\ell_1(u)\cos\frac{z(u)}{\lambda}, -\ell_2(u)\cos\frac{z(u)}{\lambda}+\ell_1(u)\sin\frac{z(u)}{\lambda}\right),
$$ 
we have
\begin{align*}
\frac{d}{du} \nu^s(u) &=\biggl(\left(\ell_1(u)\dot{z}(u)-\dot{\ell_2}(u)\right)\sin\frac{z(u)}{\lambda}-\left(\ell_2(u)\dot{z}(u)+\dot{\ell_1}(u)\right)\cos\frac{z(u)}{\lambda},\\
&\qquad\left(\ell_2(u)\dot{z}(u)+\dot{\ell_1}(u)\right)\sin\frac{z(u)}{\lambda}+\left(\ell_1(u)\dot{z}(u)-\dot{\ell_2}(u)\right)\cos\frac{z(u)}{\lambda}\biggr).
\end{align*}
We put on $\mu^s(u)=J(\nu^s (u))$, the curvature $(\ell^s, \beta^s)$ is given by $\ell^s(u)=\dot{\nu}^s(u)\cdot\mu^s(u)$ and $\beta^s(u)=\dot{s}[r[\gamma]](u)\cdot\mu^s(u)$. Therefore, we have the assertion.
\enD
Suppose that there exists $(\ell_1,\ell_2):I \to \R^2$ with $\ell_1^2(u)+\ell^2_2(u)=1$ such that $\ell_1(u) b(u) \lambda+\ell_2(u) x(u) a(u)=0$ for all $u \in I$. 
Then $(s[r[\gamma]],\nu^s):I \to \R^2 \times S^1$ is a Legendre curve by Proposition \ref{slice-curve-frontal}. 
The parallel curve of the slice curve is given by 
\begin{align*}
s[r[\gamma]]^t (u)&=s[r[\gamma]](u)+t \nu^s(u) \\
&=\Bigl((x(u)-t\ell_1(u) )\cos \frac{z(u)}{\lambda}-t \ell_2(u)\sin \frac{z(u)}{\lambda}, \\
&\qquad -\Bigl((x(u)-t\ell_1(u) )\cos \frac{z(u)}{\lambda}-t \ell_2(u)\sin \frac{z(u)}{\lambda}\Bigr)\Bigr)
\end{align*}
for $t \in \R \setminus \{0\}$. 
Suppose that there exist smooth functions $B, \tau: I \to \R$ such that 
$$
x(u)-t \ell_1(u)=B(u)\cos \tau(u), \ -t \ell_2(u)=B(u)\sin \tau(u).
$$
Then 
$$
s[r[\gamma]]^t (u)=\left(B(u)\cos \frac{z(u)-\lambda \tau(u)}{\lambda}, -B(u) \sin \frac{z(u)-\lambda \tau(u)}{\lambda}\right).
$$
We denote $\widetilde{\gamma}(u)=(B(u),z(u)-\lambda \tau(u))$. 
Then 
$$
r[\widetilde{\gamma}](u,v)=(B(u)\cos v,B(u)\sin v,z(u)-\lambda \tau(u)+\lambda v)
$$
and 
$s[r[\widetilde{\gamma}]](u)=s[r[\gamma]]^t (u)$.
Then we have the following result.
\begin{proposition}\label{relations-parallel-slice}
Under the above notations, 
if a parallel surface of the helicoidal surface and the helicoidal surface of a parallel curve of the slice curve coincide, that is, $r[\gamma]^{\widetilde{t}}(u,v)=r[\widetilde{\gamma}](u,v)$, then 
$x(u)$ is a constant and $\gamma$ is a part of a line.
\end{proposition}
\demo
By the condition $r[\gamma]^{\widetilde{t}}(u,v)=r[\widetilde{\gamma}](u,v)$, we have 
\begin{align*}
& A(u) \cos(v-\theta(u)) =B(u)\cos v,\ A(u) \sin(v-\theta(u))=B(u) \sin v, \\
& z(u)+\lambda v+\widetilde{t}k_2(u)b(u)=z(u)-\lambda \tau(u)+\lambda v.
\end{align*}
It follows that $A^2(u)=B^2(u)$ and $\widetilde{t} k_2(u)b(u)=-\lambda \tau(u)$ for all $u \in I$. 
If $A(u)=B(u)$, then $v-\theta(u)=v+2n \pi$, where $n \in \mathbb{Z}$, and hence $\theta=-2n\pi$. 
Therefore, we have $k_1(u)=0$ and $b(u)=0$. 
By \eqref{Frenet-type}, $x(u)$ is a constant and $\ell(u)=0$ for all $u \in I$, that is, $\gamma$ is a part of a line. 
If $A(u)=-B(u)$, then $v-\theta(u)=v+(2n+1)\pi$, where $n \in \mathbb{Z}$, and hence $\theta=-(2n+1)\pi$. 
Therefore, we also have $k_1(u)=0$ and $b(u)=0$. 
By \eqref{Frenet-type}, $x(u)$ is a constant and $\ell(u)=0$ for all $u \in I$, that is, $\gamma$ is a part of a line. 
\enD
In general, we consider relations between $s[r[\gamma]^{\widetilde{t}}]$ and $s[r[\gamma]]^{t}$. 
Since 
\begin{align*}
s[r[\gamma]^{\widetilde{t}}](u)&=\left(A(u) \cos \frac{z(u)+\lambda \theta(u)+\widetilde{t}k_2(u)b(u)}{\lambda},  -A(u) \sin \frac{z(u)+\lambda \theta(u)+\widetilde{t}k_2(u)b(u)}{\lambda} \right)\\
&= \begin{pmatrix}
\cos \dfrac{\widetilde{t}k_2(u)b(u)}{\lambda} & \sin \dfrac{\widetilde{t}k_2(u)b(u)}{\lambda} \\
- \sin \dfrac{\widetilde{t}k_2(u)b(u)}{\lambda} & \cos \dfrac{\widetilde{t}k_2(u)b(u)}{\lambda}
\end{pmatrix}
\begin{pmatrix}
A \cos \dfrac{z(u)+\lambda \theta(u)}{\lambda}\\
-A \sin \dfrac{z(u)+\lambda \theta(u)}{\lambda}
\end{pmatrix},
\end{align*}
if $A=B$ and $\theta=-\tau$, then 
$x+\widetilde{t}k_2a=x-t \ell_1$ and $\widetilde{t}k_1=t\ell_2$. 
Then $t=\pm \widetilde{t}\sqrt{1-k_2^2 b^2}$. 
It follows that 
$$
s[r[\gamma]^{\widetilde{t}}](u)=M(\widetilde{t},u)s[r[\gamma]]^{t(u)}(u),
$$
where 
$$
M(\widetilde{t},u)=\begin{pmatrix}
\cos \dfrac{\widetilde{t}k_2(u)b(u)}{\lambda} & \sin \dfrac{\widetilde{t}k_2(u)b(u)}{\lambda} \\
- \sin \dfrac{\widetilde{t}k_2(u)b(u)}{\lambda} & \cos \dfrac{\widetilde{t}k_2(u)b(u)}{\lambda}
\end{pmatrix}.
$$
Note that $\widetilde{t}$ is a constant, but $t$ is not a constant. 
If $k_2(u)b(u)$ is a constant, then $t$ and $M$ are a constant.

\section{Singularities of helicoidal surfaces}\label{S4}
\begin{definition}{\rm
$(1)$ Let $f$ and $g:(\R^m,0)\to(\R^n,0)$ be map-germs. 
Then $f$ is {\it $\mathcal{A}$-equivalent} to $g$ if there exist 
diffeomorphism-germs 
$\varphi:(\R^m,0)\to(\R^m,0)$ and $\Phi:(\R^n,0)\to(\R^n,0)$ 
such that $g=\Phi\circ f\circ \varphi^{-1}$ holds. 
\par
$(2)$ Let $\gamma:(I,u_0)\to(\R^2,0)$ be a curve-germ. 
We say that $\gamma$ at $u_0$ is a {\it $j/i$-cusp}, where $(i,j)=(2,3),(2,5),(3,4),(3,5)$ if $\gamma$ is $\mathcal{A}$-equivalent to the germ $t\mapsto(t^i,t^j)$ at the origin.
\par
$(3)$ Let $f:(\R^2,0)\to(\R^3,0)$ be a map-germ. 
We say that $f$ at $0$ is a {\it $j/i$-cuspidal edge}, where $(i,j)=(2,3),(2,5),(3,4),(3,5)$ if $f$ is $\mathcal{A}$-equivalent to the germ $(u,v)\mapsto(u,v^i,v^j)$ at the origin. 
A $3/2$-cuspidal edge is also called a {\it cuspidal edge}.

$(4)$ Let $f:(\R^2,0)\to(\R^3,0)$ be a map-germ,
and let $\gamma:(\R,0)\to(\R^2,0)$ be a curve-germ.
We say that $f$ at $0$ is a {\it $\gamma$-edge} if
$f$ is $\mathcal{A}$-equivalent to the germ $(u,v)\mapsto(\gamma(u),v)$ at the origin.
}
\end{definition}

For curves with $j/i$-cusps, the following criteria are known (cf. \cite{Bruce-Gaffney, Porteous}). 
\begin{proposition}\label{fact:criteria}
Let $\gamma:I\to\R^2$ be a curve and $u_0\in I$ a singular point of $\gamma$, 
namely, $\dot{\gamma}(u_0)=0$. 
Then the following assertions hold.
\par
$(1)$ $\gamma$ has a $3/2$-cusp at $u_0$ if and only if $\det(\ddot{\gamma},\dddot{\gamma})(u_0)\neq0$.
\par
$(2)$ $\gamma$ has a $5/2$-cusp at $u_0$ if and only if 
$\ddot{\gamma}(u_0)\neq0$, $\dddot{\gamma}(u_0)=k\ddot{\gamma}(u_0)$ for some constant $k\in\R$ 
and $\det(\ddot{\gamma},3\gamma^{(5)}-10k\gamma^{(4)})(u_0)\neq0$.
\par
$(3)$ $\gamma$ has a $4/3$-cusp at $u_0$ if and only if $\ddot{\gamma}(u_0)=0$ and 
$\det(\dddot{\gamma},\gamma^{(4)})(u_0)\neq0$.
\par
$(4)$ $\gamma$ has a $5/3$-cusp at $u_0$ if and only if $\ddot{\gamma}(u_0)=0$, 
$\det(\dddot{\gamma},\gamma^{(4)})(u_0)=0$ and $\det(\dddot{\gamma},\gamma^{(5)})(u_0)\neq0$.
\end{proposition}
For a curve $\gamma=(x(u),z(u))$,
if $x(u)\ne0$, then the surface 
$r=(x(u)\cos v,x(u)\sin v,z(u)+\lambda v)$
is a
$\gamma$-edge, and this can be shown by a diffeomorphism-germ
$$
\Phi(X,Y,Z)=(X\cos Z,X\sin Z,Y+\lambda Z).
$$
Thus if $x \ne 0$, singularities of the helicoidal surface $r$ are like those of the surface of revolution of a profile curve $\gamma$
(cf. \cite{Martins-Saji-Santos-Teramoto, Takahashi-Teramoto}).
We stick our consideration into the case $x(u)=0$.
We show the following.
\begin{theorem}\label{thm:criteria}
Let $\gamma(u)=(x(u),z(u))$, and let $(\gamma,\nu):(I,u_0)\to\R^2\times S^1$ 
be a Legendre curve with curvature $(\ell,\beta)$.
Assume that $x(u_0)=0$. 
Then for the helicoidal surface 
$$r(u,v)=(x(u)\cos v,x(u)\sin v,z(u)+\lambda v),$$
we have the following. 
\begin{itemize}
\item[{\rm (I)}]
If $\beta(u_0)=0, x(u_0)=0, b(u_0)\ne0$, then for any $v$,
$r$ is a $5/2$-cuspidal edge at $(u_0,v)$ if and only if 
$\dot{\beta}(u_0)\ell(u_0)\ne0$. 
In this case, $r$ at $(u_0,v)$ is never be
a cuspidal edge, a $4/3$-cuspidal edge nor a $5/3$-cuspidal edge.

\item[{\rm (II)}]
If $\beta(u_0)\not=0, (x(u_0),b(u_0))=(0,0)$, then for any $v$,
$r$ is a $3/2$-cuspidal edge at $(u_0,v)$ if and only if 
$\ell(u_0)\ne0$,
and 
$r$ is a 
$4/3$-cuspidal edge at $(u_0,v)$ if and only if $\ell(u_0)=0$ and $\dot{\ell}(u_0)\ne0$.
In this case, $r$ at $(u_0,v)$ is never be
a $5/2$-cuspidal edge nor a $5/3$-cuspidal edge.

\item[{\rm (III)}]
If $\beta(u_0)=0, (x(u_0),b(u_0))=(0,0)$, then for any $v$,
$r$ is a $5/3$-cuspidal edge at $(u_0,v)$ if and only if 
$\dot{\beta}(u_0)\ell(u_0)\ne0$. 
In this case, $r$ at $(u_0,v)$ is never be
a cuspidal edge, a $5/2$-cuspidal edge nor a $4/3$-cuspidal edge. 
\end{itemize}
\end{theorem}
\par
To show this Theorem, first we show the following lemma.
We set $x(u_0)=0$ and set $c:(\R,u_0)\to(\R^2,c(u_0))$ by
$$
c(u)=\left(x(u)\cos\dfrac{z(u)}{\lambda},x(u)\sin\dfrac{z(u)}{\lambda}\right).
$$
Since the slice curve $s[r[\gamma]]$ and $c$ are diffeomorphic, the singular points are the same. We consider $c$ in this section.
\begin{lemma}\label{lem:helieq}
For any $v_0$, the helicoidal surface $r$ at $(u_0,v_0)$ 
given in \eqref{helicoidal-surface}
is a $c$-edge.
\end{lemma}
\begin{proof}
By $\lambda\ne0$, a map-germ $\varphi:(\R^2,(u_0,v_0))\to(\R^2,(u_0,v_0))$ 
defined by
$\varphi(u,v)=(u,z(u)+\lambda v-z(u_0)-\lambda v_0+v_0)$ is a diffeomorphism-germ.
We set $(\widetilde{u},\widetilde{v})=\varphi(u,v)$.
Then
\begin{align*}
r(\varphi^{-1}(\widetilde{u},\widetilde{v}))
&=
\cos\dfrac{\widetilde w}{\lambda}
\Big(
x(\widetilde u)\cos\dfrac{z(\widetilde{u})}{\lambda},
-x(\widetilde u)\sin\dfrac{z(\widetilde{u})}{\lambda},
0\Big)\\
&\hspace{10mm}
+
\sin\dfrac{\widetilde w}{\lambda}
\Big(
x(\widetilde u)\sin\dfrac{z(\widetilde{u})}{\lambda},
x(\widetilde u)\cos\dfrac{z(\widetilde{u})}{\lambda},
0\Big)
+(0,0,\widetilde w)
\end{align*}
holds, where $\widetilde w=\widetilde v+z(u_0)+\lambda v_0-v_0$.
We set a diffeomorphism-germ 
$\Phi:(\R^3,(0,0,z(u_0)+\lambda v_0))\to(\R^3,(0,0,z(u_0)+\lambda v_0))$ by
$$
\Phi(X,Y,Z)=
\cos\dfrac{Z}{\lambda}(X,-Y,0)
+
\sin\dfrac{Z}{\lambda}(Y,X,0)
+(0,0,Z).
$$
Then we see
$\Phi^{-1}(r(\varphi^{-1}(\widetilde{u},\widetilde{v})))
=(c(\widetilde{u}),\widetilde{v}+z(u_0)+\lambda v_0-v_0)$.
Since $(c(\widetilde{u}),\widetilde{v}+z(u_0)+\lambda v_0-v_0)$ 
at $(\widetilde{u}_0,\widetilde{v}_0)=(u_0,v_0)$ is a $c$-edge,
this shows the assertion.
\end{proof}
By Lemma \ref{lem:helieq}, to study singularities of $r$,
it is enough to see the curve
$c(u)$.
From now on, we write $c$ as a column vector,
and we omit the variable $(u)$ of functions.
By \eqref{Frenet-type} and
\begin{equation}\label{eq:cossin}
\dfrac{d}{du}\pmt{\cos z/\lambda\\\sin z/\lambda}
=
\pmt{-\dfrac{\dot z}{\lambda}\sin z/\lambda\\[2mm] \dfrac{\dot z}{\lambda}\cos z/\lambda}
=
\dfrac{\beta a}{\lambda}\pmt{-\sin z/\lambda\\ \cos z/\lambda}
=
\dfrac{\beta a}{\lambda}\pmt{0&-1\\ 1&0}\pmt{\cos z/\lambda\\ \sin z/\lambda},
\end{equation}
it holds that
\begin{equation}\label{eq:gammap}
\dot c=
\beta\pmt{-b\cos z/\lambda-xa/\lambda\sin z/\lambda\\
 -b\sin z/\lambda+xa/\lambda\cos z/\lambda}\\
=
\beta\pmt{-b&-xa/\lambda\\ xa/\lambda&-b}
\pmt{\cos z/\lambda\\\sin z/\lambda}.
\end{equation}
Thus $\dot c(u_0)=0$ if and only if
\begin{itemize}
\item
$\beta(u_0)=0, x(u_0)=0, b(u_0)\not=0$, 
\item
$\beta(u_0)\not=0, (x(u_0),b(u_0))=(0,0)$ or
\item
$\beta(u_0)=0, (x(u_0),b(u_0))=(0,0)$.
\end{itemize}
Let $u=u_0$ be a singular point of $c$, then 
to show Theorem \ref{thm:criteria},
it is enough to prove the following proposition.
\begin{proposition}\label{thm:singcond}
Under the above notations, we have the following.
\begin{itemize}
\item[{\rm (I)}]
If $\beta(u_0)=0, x(u_0)=0, b(u_0)\ne0$, then 
$\ddot{c}(u_0)\ne0$ if and only if $\dot{\beta}(u_0)\ne0$,
and
$c$ is a $5/2$-cusp at $u_0$ if and only if 
$\dot{\beta}(u_0)\ell(u_0)\ne0$. 
The curve $c$ satisfies 
$\ddot{c}(u_0)=0$, $\dddot{c}(u_0)\ne0$
if and only if $\dot{\beta}(u_0)=0$, $\ddot{\beta}(u_0)=0$.
On the other hand, $c$ at $u_0$ never be
a cusp, a $4/3$-cusp nor a $5/3$-cusp.

\item[{\rm (II)}]
If $\beta(u_0)\not=0, (x(u_0),b(u_0))=(0,0)$, then
$\ddot{c}(u_0)\ne0$ if and only if $\ell(u_0)\ne0$,
in this case, $c$
is a $3/2$-cusp at $u_0$.
The curve $c$ satisfies
$\ddot{c}(u_0)=0$, $\dddot{c}(u_0)\ne0$
if and only if $\ell(u_0)=0$ and $\dot{\ell}(u_0)\ne0$,
in this case, $c$ at $u_0$ is a $4/3$-cusp.
On the other hand, $c$ at $u_0$ never be
a $5/2$-cusp nor a $5/3$-cusp.

\item[{\rm (III)}]
If $\beta(u_0)=0, (x(u_0),b(u_0))=(0,0)$, then
$\ddot{c}(u_0)=0$ holds.
The curve $c$ satisfies
$\ddot{c}(u_0)=0$, $\dddot{c}(u_0)\ne0$
$\dot{\beta}(u_0)\ell(u_0)\ne0$,
in this case, $c$ at $u_0$ is a $5/3$-cusp. 
On the other hand, $c$ at $u_0$ never be
a cusp, a $5/2$-cusp nor a $4/3$-cusp. 
\end{itemize}
\end{proposition}
\par
To show this proposition, we calculate the derivatives of $c$.
The following lemma holds.
\begin{lemma}\label{lem:gammamat}
The $n$-th derivative $c^{(n)}$ has the form
$$c^{(n)}=
\pmt{
C^n_{11}&C^n_{12}\\
C^n_{21}&C^n_{22}}
\pmt{\cos z/\lambda\\ \sin z/\lambda},
$$
and it holds that
$C^n_{11}=C^n_{22}$ and
$C^n_{12}=-C^n_{21}$.
\end{lemma}
\begin{proof}
We saw the case $n=1$ in \eqref{eq:gammap}.
We assume that the assertion holds for the $(n-1)$-st derivative.
By \eqref{eq:cossin}, we see
\begin{align*}
c^{(n)}&
=
\pmt{
(C^{n-1}_{11})'&(C^{n-1}_{12})'\\
(C^{n-1}_{21})'&(C^{n-1}_{22})'}
\pmt{\cos z/\lambda\\ \sin z/\lambda}
+
\dfrac{\beta a}{\lambda}
\pmt{
C^{n-1}_{12}&-C^{n-1}_{11}\\
C^{n-1}_{22}&-C^{n-1}_{21}}
\pmt{\cos z/\lambda\\ \sin z/\lambda}\\
&=
\pmt{
(C^{n-1}_{11})'+\dfrac{\beta a}{\lambda}
C^{n-1}_{12}&(C^{n-1}_{12})'-\dfrac{\beta a}{\lambda}
C^{n-1}_{11}\\[2mm]
(C^{n-1}_{21})'+\dfrac{\beta a}{\lambda}
C^{n-1}_{22}&(C^{n-1}_{22})'-\dfrac{\beta a}{\lambda}
C^{n-1}_{21}}
\pmt{\cos z/\lambda\\ \sin z/\lambda},
\end{align*}
where $(~)'=\dot{~}$.
By this formula and by the induction hypothesis, we see the assertion.
\end{proof}
We set $C_1^n=\trans{(C^n_{11},C^n_{21})}$, where $\trans(~)$ stands for
the matrix transposition.
By this lemma, it is enough to state $C_1^n$ for $c^{(n)}$.
We turn to calculate the derivatives of $c$.
Since 
$$
(C_1^1)'=
\dot\beta \pmt{-b\\ -xa/\lambda}+
\beta\pmt{-\ell a\\ -b(\beta a+\ell x)/\lambda}
$$
together with \eqref{eq:cossin}, we have
\begin{equation}\label{eq:c2}
C_1^2=
\pmt{
-b\dot\beta-\ell a\beta-\beta^2 a^2x/\lambda^2
\\
-(-xa\dot\beta+2ab\beta^2+\ell xb\beta)/\lambda}.
\end{equation}
By
\begin{align*}
(C^2_{11})'
&=
-2\ell a\dot\beta-b\ddot\beta-\dot\ell a\beta+\ell ^2b\beta
-(2\beta\dot\beta a^2x-2\beta^2 ab\ell x-\beta^3 a^2b)/\lambda^2,\\
(C^2_{21})'
&=
-(5ab\beta\dot\beta+2xb\ell \dot\beta-xa\ddot\beta
-3b^2\ell \beta^2+2a^2\ell \beta^2
+\dot\ell xb\beta+\ell xa\ell \beta)/\lambda,
\end{align*}
we see
\begin{equation}\label{eq:c3}
C_1^3=\pmt{
-2\ell a\dot\beta-b\ddot\beta-\dot\ell a\beta+\ell ^2b\beta
+\dfrac{1}{\lambda^2}\Big(-3\beta\dot\beta a^2x+3\beta^2 ab\ell x+3\beta^3 a^2b\Big)
\\[4mm]
\dfrac{1}{\lambda}\Big(-6ab\beta\dot\beta-2xb\ell \dot\beta+xa\ddot\beta
+3b^2\ell \beta^2-3a^2\ell \beta^2
-\dot\ell xb\beta-xa\ell ^2\beta\Big)
-\dfrac{\beta^3 a^3x}{\lambda^3}
}.
\end{equation}
Furthermore, by
\begin{align*}
(C^3_{11})'
&=
-3\dot\ell a\dot\beta+3\ell ^2b\dot\beta-3\ell a\ddot\beta
-b\dddot\beta-\ddot\ell a\beta+3\ell \dot\ell b\beta
+\ell ^3a\beta \\
&\hspace{10mm}+
\dfrac{1}{\lambda^2}
\Big(-3(\dot\beta)^2a^2x-3\beta\ddot\beta a^2x+12\beta\dot\beta a\ell bx
+12\beta^2\dot\beta a^2b-3\beta^2b^2\ell ^2x+3\beta^2a^2\ell ^2x\\
&\hspace{10mm}+3\beta^2ab\dot\ell x-
9\beta^3ab^2\ell +3\beta^3a^3\ell \Big),\\
(C^3_{21})'
&=
\dfrac{1}{\lambda}\Big(14\ell b^2\beta\dot\beta -12a^2\ell \beta\dot\beta 
-6ab(\dot\beta)^2-7ab\beta\ddot\beta-3\ell ^2ax\dot\beta-3xb\dot\ell \dot\beta
-3xb\ell \ddot\beta +xa\dddot\beta\\
&\hspace{10mm}+13ab\ell ^2\beta^2
+
4b^2\dot\ell \beta^2-3a^2\dot\ell \beta^2-\ddot\ell xb\beta
-3\dot\ell x\ell a\beta+x\ell ^3\beta b\Big)\\
&\hspace{10mm}
-\dfrac{1}{\lambda^3}\Big(3\beta^2\dot\beta a^3x-3\beta^3a^2x\ell b-\beta^4a^3b\Big),
\end{align*}
we have
\begin{equation}\label{eq:c4}
C_{1}^4=
\pmt{
\begin{array}{l}
-3\dot\ell a\dot\beta +3\ell ^2b\dot\beta
-3\ell a\ddot\beta -b\dddot\beta-\ddot\ell a\beta+3\ell \dot\ell b\beta \\
\hspace{10mm}+\dfrac{\ell ^3a\beta}{\lambda^2}\Big(-3(\dot\beta)^2a^2x
-4\beta\ddot\beta a^2x+14\beta\dot\beta a\ell bx
+18\beta^2\dot\beta a^2b-3\beta^2b^2\ell ^2x\\
\hspace{10mm}+4\beta^2a^2\ell ^2x+4\beta^2ab\dot\ell x
-12\beta^3ab^2\ell +6\beta^3a^3\ell \Big)
+\dfrac{1}{\lambda^4}\beta^4 a^4x\\[5mm]
\dfrac{1}{\lambda}\Big(14\ell b^2\beta\dot\beta-14a^2\ell \beta\dot\beta
-6ab(\dot\beta)^2
-8ab\beta\ddot\beta -3\ell ^2ax\dot\beta-3xb\dot\ell \dot\beta
-3xb\ell \ddot\beta\\
\hspace{10mm}+xa\dddot\beta
+14ab\ell ^2\beta^2+4b^2\dot\ell \beta^2-4a^2\dot\ell \beta^2-\ddot\ell xb\beta
-3\dot\ell x\ell a\beta+x\ell ^3\beta b\Big)\\
\hspace{10mm}
+\dfrac{1}{\lambda^3}\Big(-6\beta^2\dot\beta a^3x
+4\beta^4a^3b+6\beta^3a^2x\ell b\Big)
\end{array}
}.\end{equation}
Finally, differentiating $C_{1}^4$,
we have
\begin{equation}\label{eq:c5}
C_{1}^5=
\pmt{
\begin{array}{l}
-b \beta \ell ^4+4 a \ell ^3 \dot\beta+6 a \beta \ell ^2 \dot\ell 
+12 b \ell  \dot\beta \dot\ell +3 b \beta (\dot\ell)^2+6 b \ell ^2 \ddot\beta\\
\hspace{5mm}
-6 a \dot\ell \ddot\beta+4 b \beta \ell  \ddot\ell -4 a \dot\beta \ddot\ell
-4 a \ell  \dddot\beta-a \beta \dddot\ell -b \beta^{(4)}\\
\hspace{5mm}+
\dfrac{1}{\lambda^2}
\Big(-60 a^2 b \beta^3 \ell ^2+15 b^3 \beta^3 \ell ^2
+50 a^3 \beta^2 \ell  \dot\beta
-100 a b^2 \beta^2 \ell  \dot\beta+45 a^2 b \beta (\dot\beta)^2\\
\hspace{5mm}+10 a^3 \beta^3 \dot\ell 
-20 a b^2 \beta^3 \dot\ell +30 a^2 b \beta^2 \ddot\beta\Big)
-
\dfrac{5 a^4 b \beta^5}{\lambda^4}
+*x
\\[5mm]
\dfrac{1}{\lambda}
\Big(15 a^2 \beta^2 \ell ^3-15 b^2 \beta^2 \ell ^3+90 a b \beta \ell ^2 \dot\beta
-20 a^2 \ell  (\dot\beta)^2+20 b^2 \ell  (\dot\beta)^2+50 a b \beta^2 \ell  \dot\ell \\
\hspace{5mm}-25 a^2 \beta \dot\beta \dot\ell +25 b^2 \beta \dot\beta \dot\ell 
-25 a^2 \beta \ell  \ddot\beta
+25 b^2 \beta \ell  \ddot\beta-20 a b \dot\beta \ddot\beta-5 a^2 \beta^2 \ddot\ell \\
\hspace{5mm}+5 b^2 \beta^2 \ddot\ell-10 a b \beta \dddot\beta\Big)
+
\dfrac{1}{\lambda^3}
\Big(10 a^4 \beta^4 \ell -30 a^2 b^2 \beta^4 \ell +40 a^3 b \beta^3 \dot\beta\Big)
+*x
\end{array}
},
\end{equation}
where $*$ stands for a function.
Since we will see the derivatives up to $5$-th degree, 
the terms with $x$ in $C_{1}^5$ will not be necessary in the later calculations.
On the other hand, we note that
\begin{equation}\label{eq:detlemma}
\det
\left(
\pmt{
x_{11}&-x_{21}\\
x_{21}&x_{11}}
\pmt{v_1\\v_2},\ 
\pmt{y_{11}&-y_{21}\\
y_{21}&y_{11}}
\pmt{v_1\\v_2}
\right)
=
(v_1^2+v_2^2)
\det
\pmt{
 x_{11}&y_{11}\\
 x_{21}&y_{21}}.
\end{equation}
\begin{proof}[Proof of Proposition {\rm \ref{thm:singcond}}]
{\rm (I)}
We assume $\beta(u_0)=x(u_0)=0$ and $b(u_0)\ne0$. Then it holds that
$C_1^2=\trans{(-b\dot\beta,0)}$ and
$C_1^3=\trans{(-2\ell a\dot\beta-b\ddot\beta,0)}$ at $u_0$.
Thus $\det(\ddot c,\dddot c)(u_0)=0$.
On the other hand, $\ddot c(u_0)=0$ if and only if
$\dot \beta (u_0)=0$.
We assume $\dot \beta (u_0)=0$.
Then
$C_1^3=\trans{(-b\ddot \beta ,0)}$
and
$C_1^4=\trans{(-3\ell a\ddot \beta -b\dddot \beta ,0)}$ at $u_0$.
Thus $\det(\dddot c,c^{(4)})(u_0)=0$.
Next, we assume $\dot \beta (u_0)\ne0$.
Setting 
$$k=
\dfrac{2a \ell  \dot \beta +b \ddot \beta }{b \dot \beta },
$$
we have $\dddot c=k\ddot c$ at $u_0$.
Then by a direct calculation,
$\det(\ddot c,3c^{(5)}-10kc^{(4)})(u_0)=0$ if and only if
$
(a^2+b^2) \ell  (\dot \beta )^2=0.
$
Thus we see the assertion.

{\rm (II)}
We assume $\beta(u_0)\not=0$ and $(x(u_0),b(u_0))=(0,0)$. 
By \eqref{eq:c2} and \eqref{eq:c3},
we have
$$
C_1^2=\pmt{-\ell a\beta\\0},\quad
C_1^3=\pmt{-2\ell a\dot \beta -\dot\ell a\beta\\-3a^2\ell \beta^2/\lambda}.
$$
Thus $\ddot c\ne0$ if and only if $\ell (u_0)\ne0$.
Moreover, since
$$
\det(\ddot{c},\dddot{c})(u_0)=
\dfrac{3\ell^2(u_0)a^3(u_0)\beta^3(u_0)}{\lambda},
$$
$c$ is a $3/2$-cusp at $u_0$ if and only if $\ell (u_0)\ne0$.
This also implies 
$c$ does not have a $5/2$-cusp.
If $\ell (u_0)\ne0$, then by \eqref{eq:c3} and \eqref{eq:c4},
we have
$$
C_1^3
=\pmt{
-\dot\ell a\beta\\0},\quad
C_1^4
=
\pmt{
-3\dot\ell a\dot \beta -\ddot\ell a\beta\\
-4a^2\dot \ell  \beta^2/\lambda}.
$$
Thus $\dddot c(u_0)\ne0$ if and only if $\dot \ell  (u_0)\ne0$, and
we have
$\det(\dddot c(u_0),c^{(4)}(u_0))=(4(\dot \ell  )^3a^3\beta^3/\lambda)(u_0)$.
Thus $c$ is a $4/3$-cusp at $u_0$ if and only if $\ell =0$, $\dot \ell  \ne0$ at $u_0$.
This also implies 
$c$ never be a $5/3$-cusp.

{\rm (III)}
We assume $\beta(u_0)=0, (x(u_0),b(u_0))=(0,0)$.
Then we see
$\ddot c=0$ and
$\det(\dddot{c},c^{(4)})=0$ at $u_0$.
Thus
$c$ is never be a $3/2$-cusp, a $5/2$-cusp nor a $4/3$-cusp.
Moreover by
$$
\det(\dddot{c},c^{(5)})(u_0)=\dfrac{40\dot{\ell}^2(u_0)a^3(u_0)\dot{\beta}^3(u_0)}{\lambda},
$$
$c$ is a $5/3$-cusp at $u_0$
if and only if $\ell(u_0)\dot{\beta}(u_0)\ne0$.
\end{proof}

\section{Examples}

We give examples of helicoidal surfaces of frontals with singular points
dealt with the above.
We consider all of the singular points are $u=0$.

\begin{example}
Let $(\gamma,\nu): (\R,0) \to \R^2 \times S^1$ be 
\begin{align*}
\gamma(u)&=(x(u),z(u))=(u^2+u^3,u^2), \\
\nu(u) &=(a(u),b(u))=\left(\dfrac{2}{\sqrt{8+12u+9u^2}}, -\dfrac{2+3u}{\sqrt{8+12u+9u^2}}\right) 
\end{align*} 
and $\lambda=1/2$.
Then $(\gamma,\nu)$ is a Legendre curve with curvature
$$\ell(u)=-\dfrac{6}{8+12u+9u^2},\quad \beta(u)=u\sqrt{8+12u+9u^2}.$$
If we take $(k_1,k_2): (\R,0) \to S^1$,
\begin{align*}
k_1(u)&=\frac{2+3u}{\sqrt{(2+3u)^2+4(u^2+u^3)^2(8+12u+9u^2)}}, \\
k_2(u)&=-\frac{2(u^2+u^3)\sqrt{8+12u+9u^2}}{\sqrt{(2+3u)^2+4(u^2+u^3)^2(8+12u+9u^2)}},
\end{align*}
then the helicoidal surface $(r,\bn,\bs)$ is a framed surface by Proposition \ref{basic-invariants-framed-surface}.
By a direct calculation, we have 
$\beta(0)=x(0)=0$, $b(0)=-1/\sqrt{2}\ne0$, 
$\ell(0)\dot{\beta}(0)=- 3/\sqrt{2}\ne0$.
By Theorem \ref{thm:criteria}, it holds that
$r$ at $(0,0)$ is a $5/2$-cuspidal edge. 
See Figure \ref{fig:5-2cusp}, where we show
the curve $\gamma$ with a thick line with the $x$ and $z$-axes with thin lines,
the helicoidal surface $r$ and
a closer view of singular points of $r$.
\begin{figure}[htbp]
\centering
\includegraphics[width=0.25\linewidth]{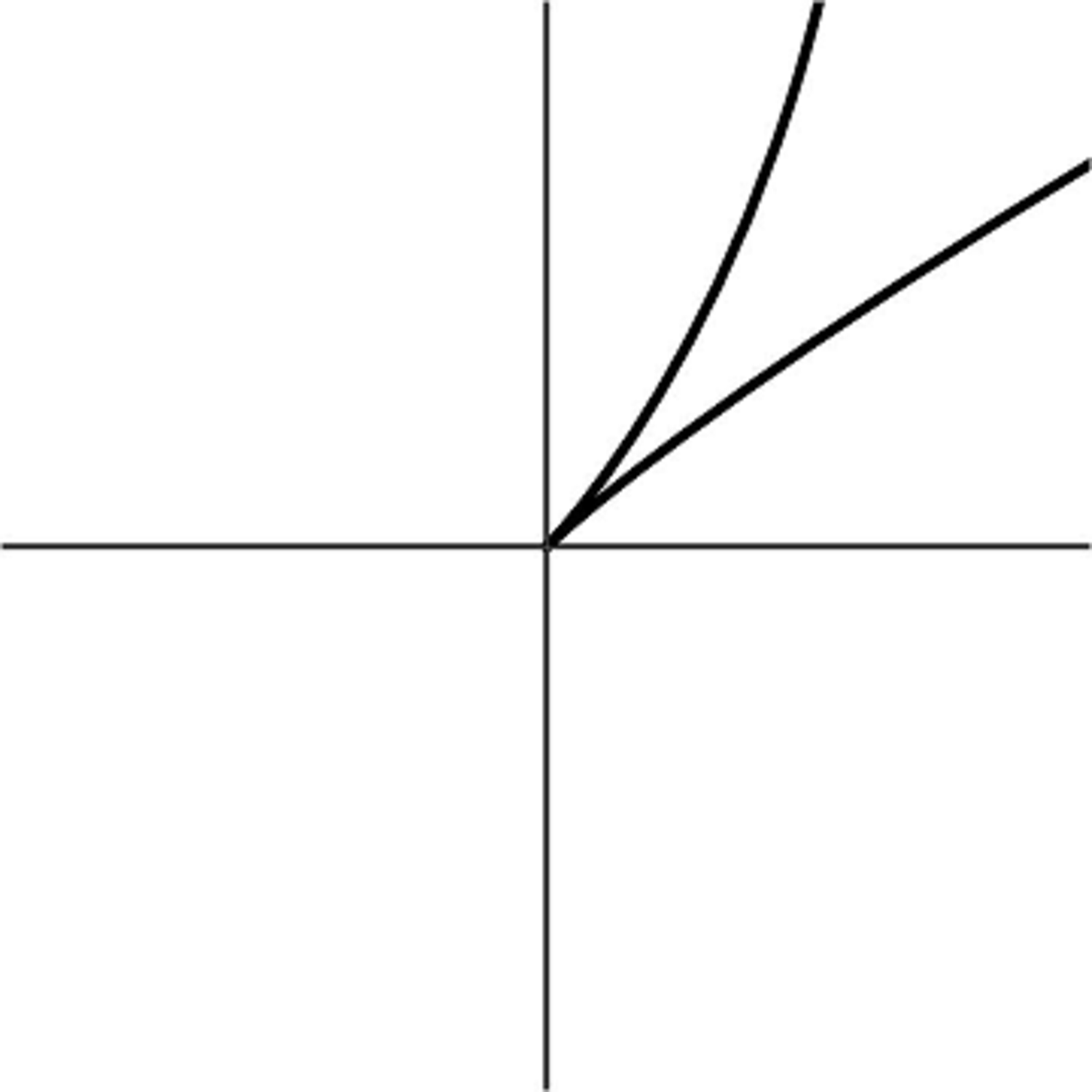}\hspace{10mm}
\includegraphics[width=0.25\linewidth]{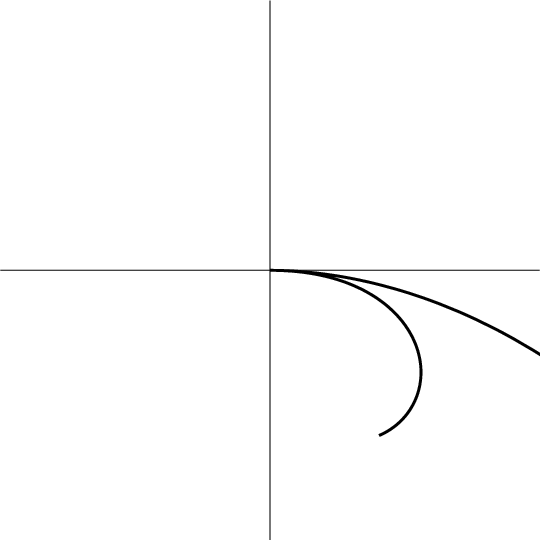}\hspace{10mm}\\
\includegraphics[width=0.1\linewidth]{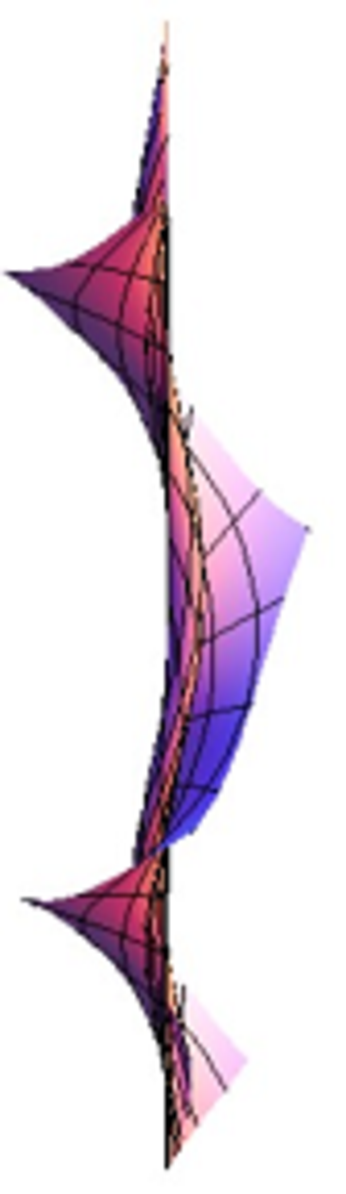}\hspace{20mm}
\includegraphics[width=0.1\linewidth]{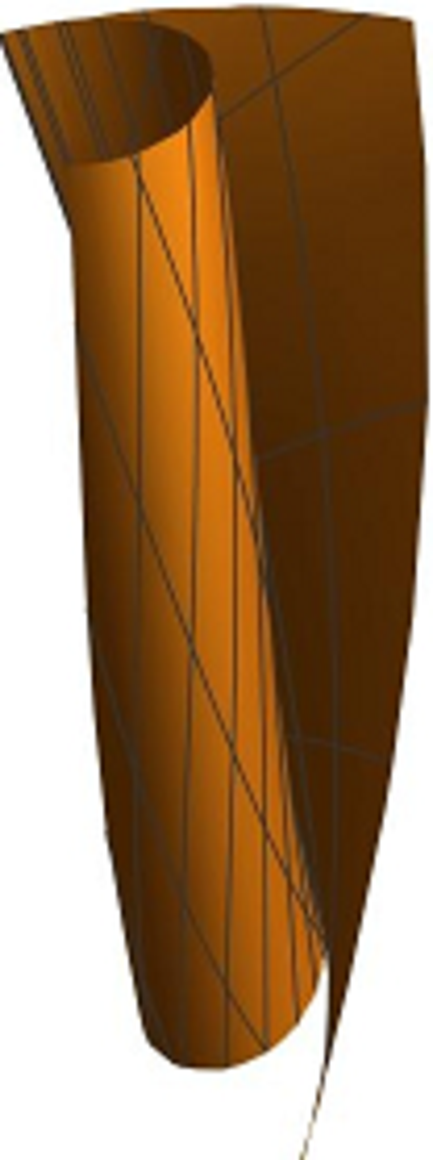}
\label{fig:5-2cusp}
\caption{Helicoidal surface with $5/2$-cuspidal edge 
(the curve $\gamma$, the slice curve $s[r[\gamma]]$, the surface $r$, closer view)}
\end{figure}
\end{example}

\begin{example}
Let $(\gamma,\nu): (\R,0) \to \R^2 \times S^1$ be 
\begin{align*}
\gamma(u)&=(x(u),z(u))=(u^2,u), \\
\nu(u) &=(a(u),b(u))=\left(\dfrac{1}{\sqrt{1+4u^2}}, -\dfrac{2u}{\sqrt{1+4u^2}}\right) 
\end{align*} 
and $\lambda=1/2$.
Then $(\gamma,\nu)$ is a Legendre curve with curvature
$$\ell(u)=-\dfrac{2}{{1+4u^2}},\quad  \beta(u)=\sqrt{1+4u^2}.$$
If we take $(k_1,k_2): (\R,0) \to S^1$,
\begin{align*}
k_1(u)=\frac{1}{\sqrt{1+u^2(1+4u^2)}}, \ 
k_2(u)=-\frac{u\sqrt{1+4u^2}}{1+u^2(1+4u^2)},
\end{align*}
then the helicoidal surface $(r,\bn,\bs)$ is a framed surface by Proposition \ref{basic-invariants-framed-surface}.
By a direct calculation, we have 
$\beta(0)=1\ne0$, $(x(0),b(0))=(0,0)$, $\ell(0)=-2\ne0$.
By Theorem \ref{thm:criteria}, it holds that
 $r$ at $(0,0)$ is a $3/2$-cuspidal edge. See Figure \ref{fig:3-2cusp}.
\begin{figure}[htbp]
\centering
\includegraphics[width=0.25\linewidth]{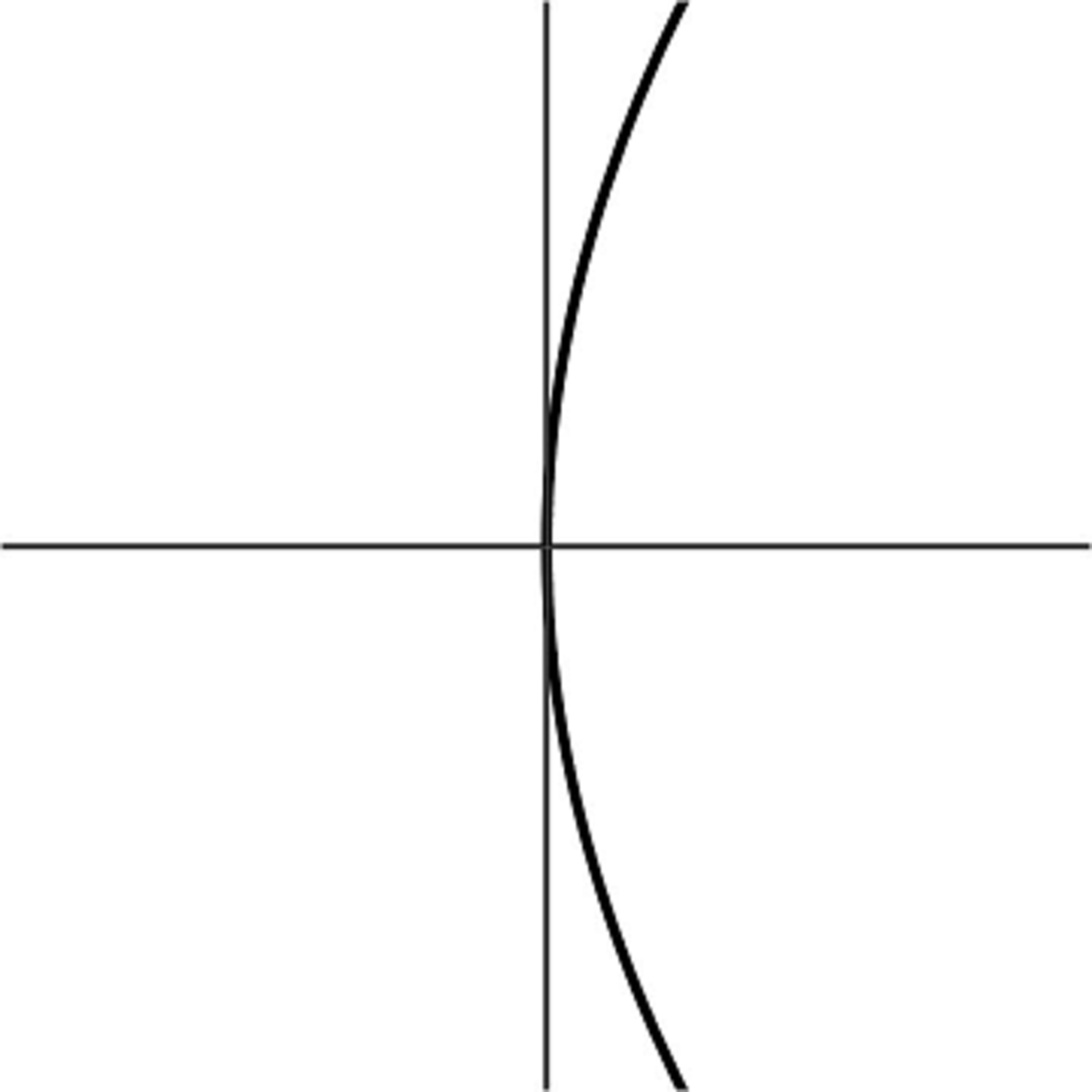}\hspace{10mm}
\includegraphics[width=0.25\linewidth]{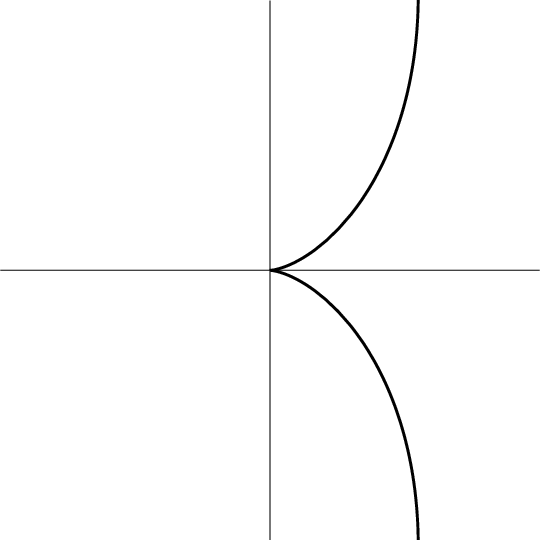}\hspace{10mm}\\
\includegraphics[width=0.16\linewidth]{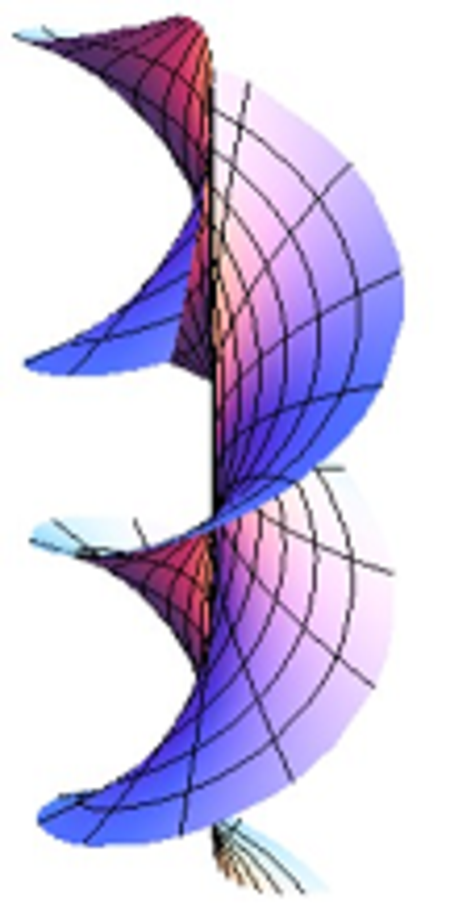}\hspace{20mm}
\includegraphics[width=0.16\linewidth]{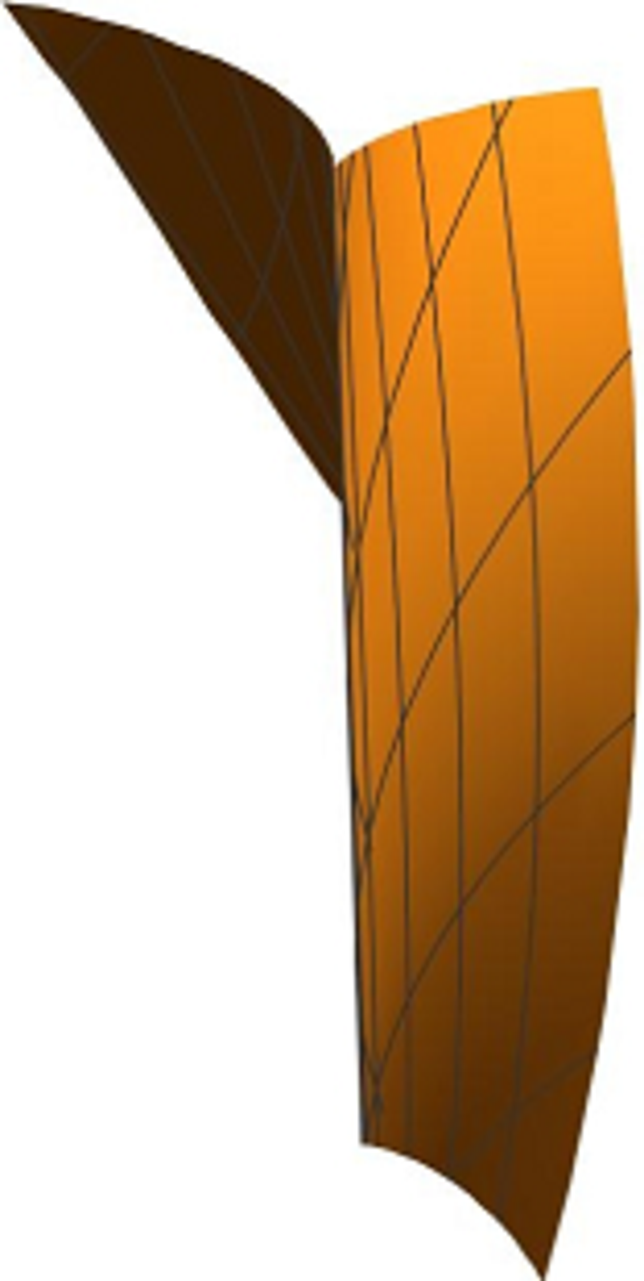}
\label{fig:3-2cusp}
\caption{Helicoidal surface with $3/2$-cuspidal edge 
(the curve $\gamma$, the slice curve $s[r[\gamma]]$, the surface $r$, closer view)}
\end{figure}
\end{example}

\begin{example}
Let $(\gamma,\nu): (\R,0) \to \R^2 \times S^1$ be 
\begin{align*}
\gamma(u)&=(x(u),z(u))=(u^3,u), \\
\nu(u) &=(a(u),b(u))=\left(\dfrac{1}{\sqrt{1+9u^4}}, -\dfrac{3u^2}{\sqrt{1+9u^4}}\right) 
\end{align*} 
and $\lambda=1/2$.
Then $(\gamma,\nu)$ is a Legendre curve with curvature
$$\ell(u)=-\dfrac{6u}{{1+9u^4}},\quad  \beta(u)=\sqrt{1+9u^4}.$$
If we take $(k_1,k_2): (\R,0) \to S^1$,
\begin{align*}
k_1(u)=\frac{3}{\sqrt{9+4u^2(1+9u^4)}}, \ 
k_2(u)=-\frac{2u\sqrt{1+9u^4}}{9+4u^2(1+9u^4)},
\end{align*}
then the helicoidal surface $(r,\bn,\bs)$ is a framed surface by Proposition \ref{basic-invariants-framed-surface}.
By a direct calculation, we have 
$\beta(0)=1\ne0$, $(x(0),b(0))=(0,0)$, $\ell(0)=0$, $\dot{\ell}(0)=-6\ne0$.
By Theorem \ref{thm:criteria}, it holds that 
$r$ at $(0,0)$ is a $4/3$-cuspidal edge. See Figure \ref{fig:4-3cusp}.
\begin{figure}[htbp]
\centering
\includegraphics[width=0.25\linewidth]{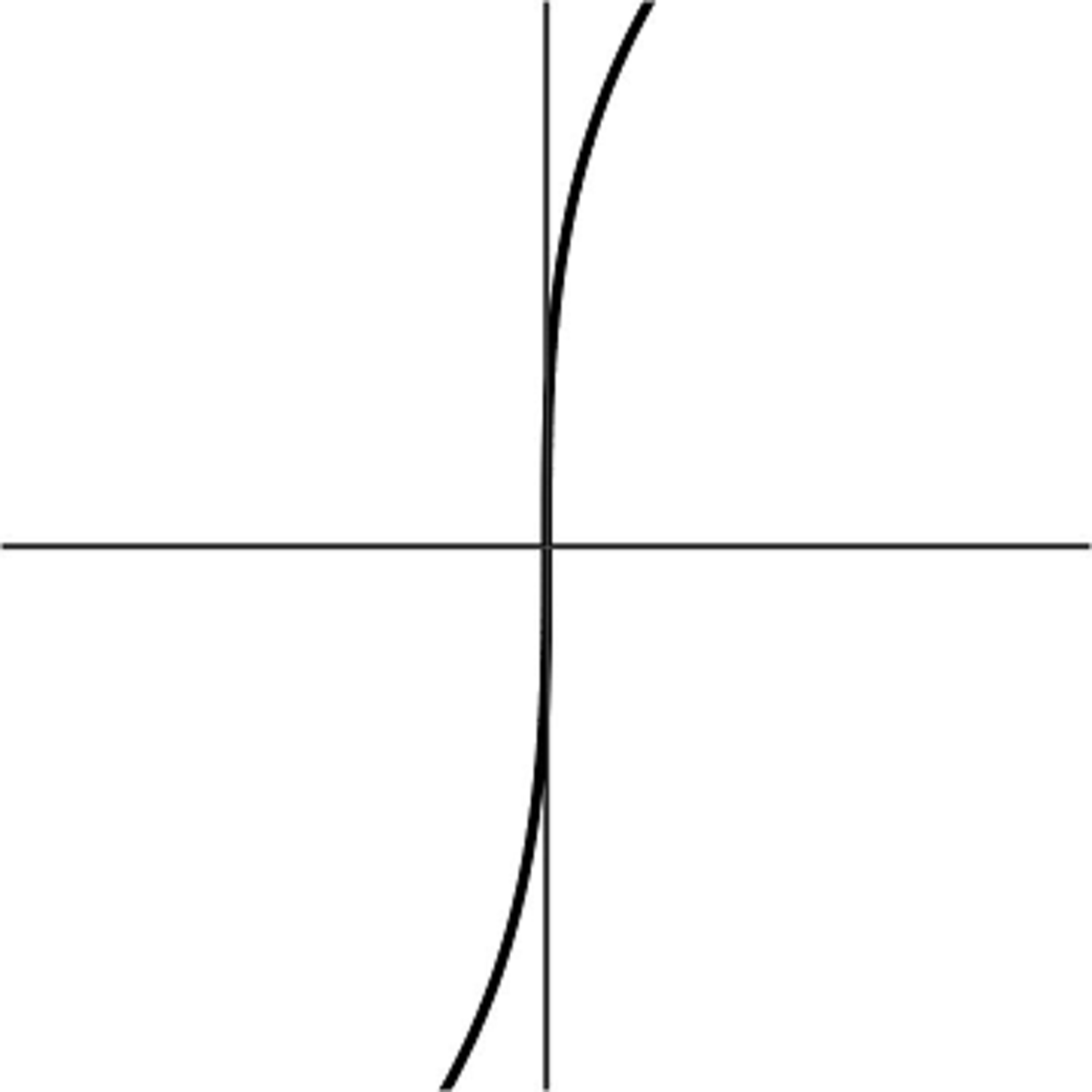}\hspace{10mm}
\includegraphics[width=0.25\linewidth]{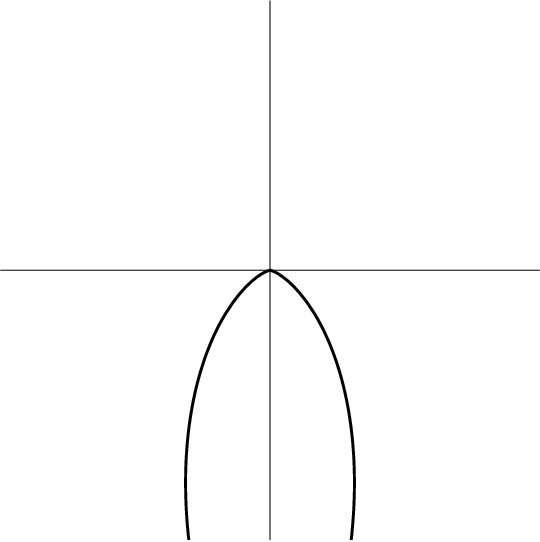}\hspace{10mm}\\
\includegraphics[width=0.1\linewidth]{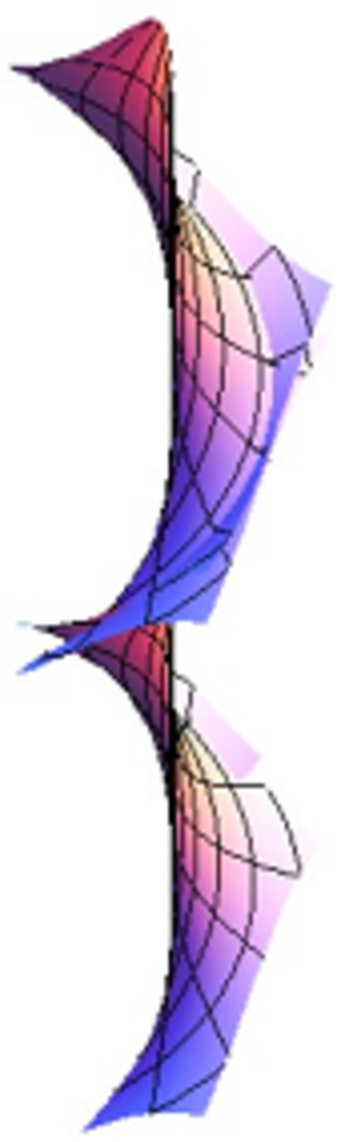}\hspace{20mm}
\includegraphics[width=0.17\linewidth]{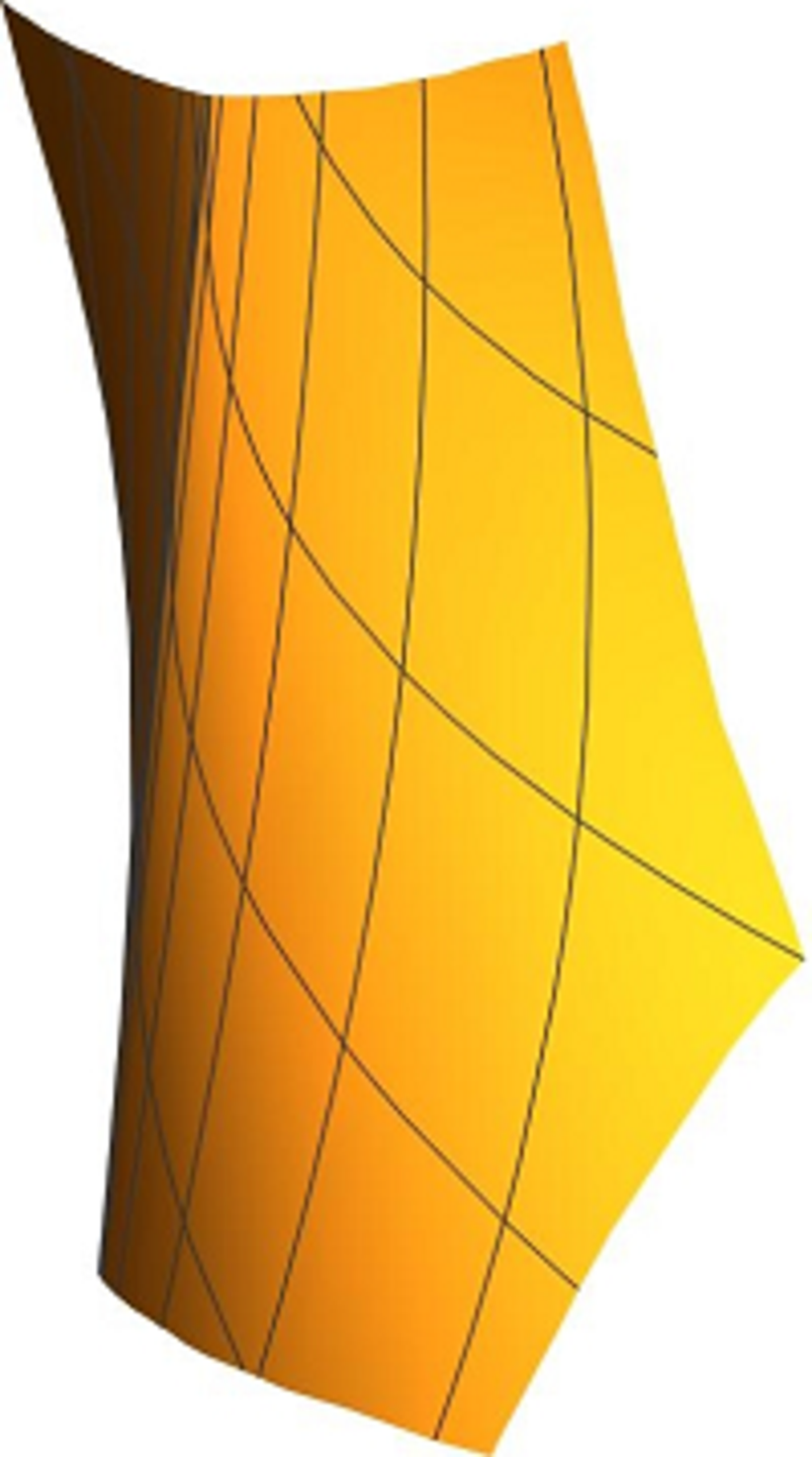}
\label{fig:4-3cusp}
\caption{Helicoidal surface with $4/3$-cuspidal edge 
(the curve $\gamma$, the slice curve $s[r[\gamma]]$, the surface $r$, closer view)}
\end{figure}
\end{example}

\begin{example}
Let $(\gamma,\nu): (\R,0) \to \R^2 \times S^1$ be 
\begin{align*}
\gamma(u)&=(x(u),z(u))=(u^3,u^2), \\
\nu(u) &=(a(u),b(u))=\left(\dfrac{2}{\sqrt{4+9u^2}}, -\dfrac{3u}{\sqrt{4+9u^2}}\right) 
\end{align*} 
and $\lambda=1/2$.
Then $(\gamma,\nu)$ is a Legendre curve with curvature
$$\ell(u)=-\dfrac{6}{{4+9u^2}},\quad  \beta(u)=u \sqrt{4+9u^2}.$$
If we take $(k_1,k_2): (\R,0) \to S^1$,
\begin{align*}
k_1(u)=\frac{3}{\sqrt{9+4u^4(4+9u^2)}}, \ 
k_2(u)=-\frac{2u^2\sqrt{4+9u^2}}{9+4u^4(4+9u^2)},
\end{align*}
then the helicoidal surface $(r,\bn,\bs)$ is a framed surface by Proposition \ref{basic-invariants-framed-surface}.
By a direct calculation, we have 
$\beta(0)=0$, $(x(0),b(0))=(0,0)$, $\ell(0)\dot{\beta}(0)=-3\ne0$.
By Theorem \ref{thm:criteria}, it holds that
 $r$ at $(0,0)$ is a $5/3$-cuspidal edge. See Figure \ref{fig:5-3cusp}.
\begin{figure}[htbp]
\centering
\includegraphics[width=0.25\linewidth]{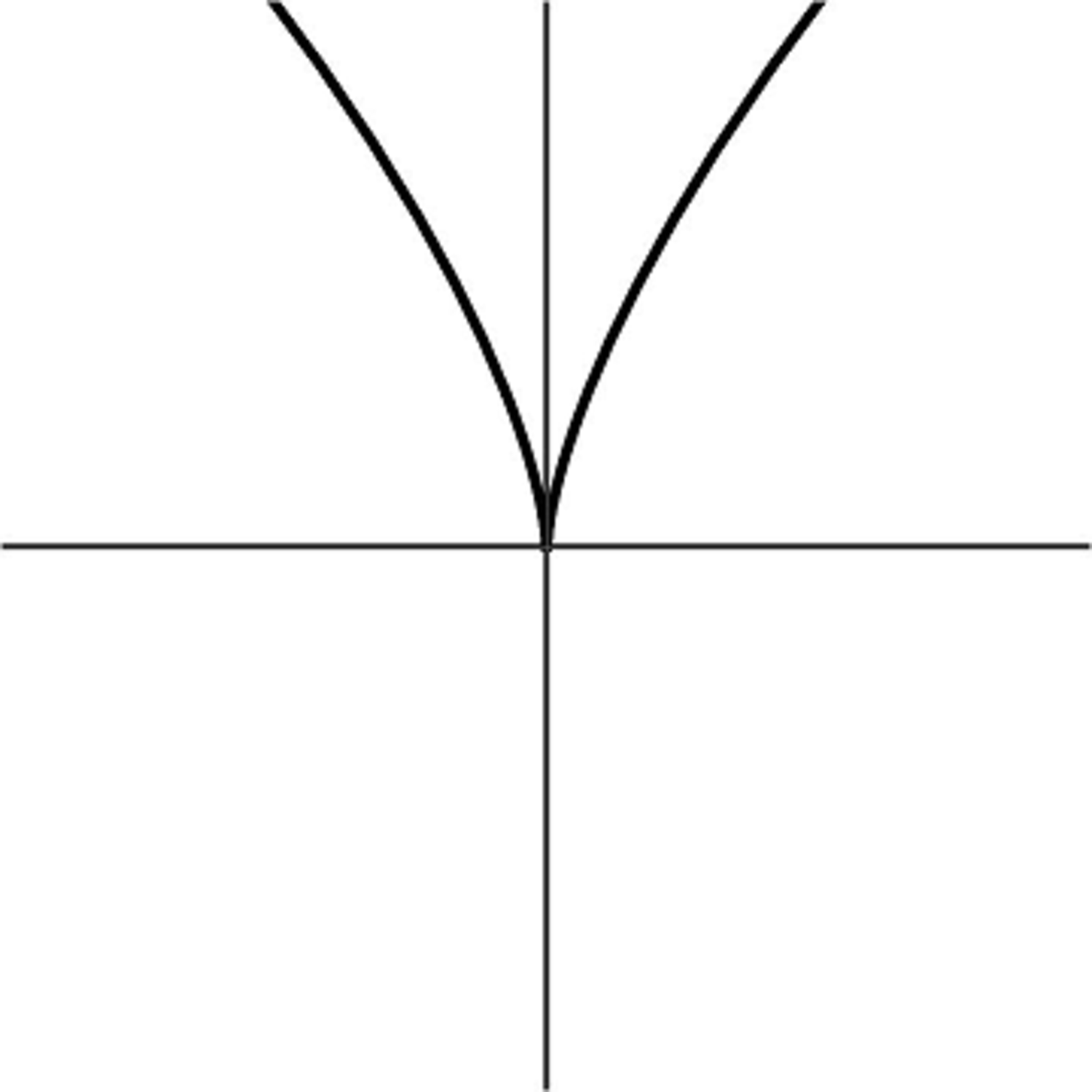}\hspace{10mm}
\includegraphics[width=0.25\linewidth]{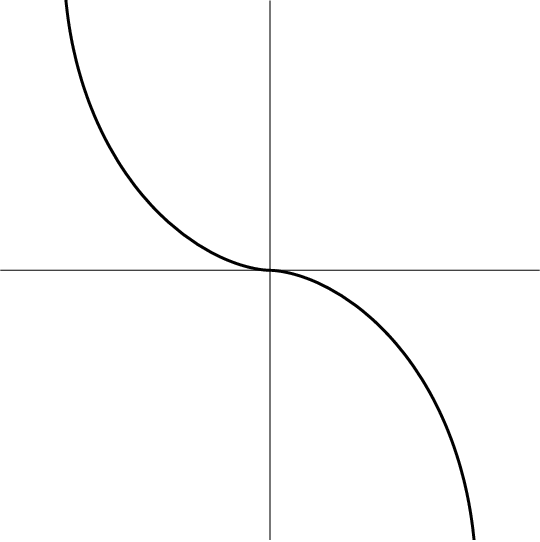}\hspace{10mm}\\
\includegraphics[width=0.11\linewidth]{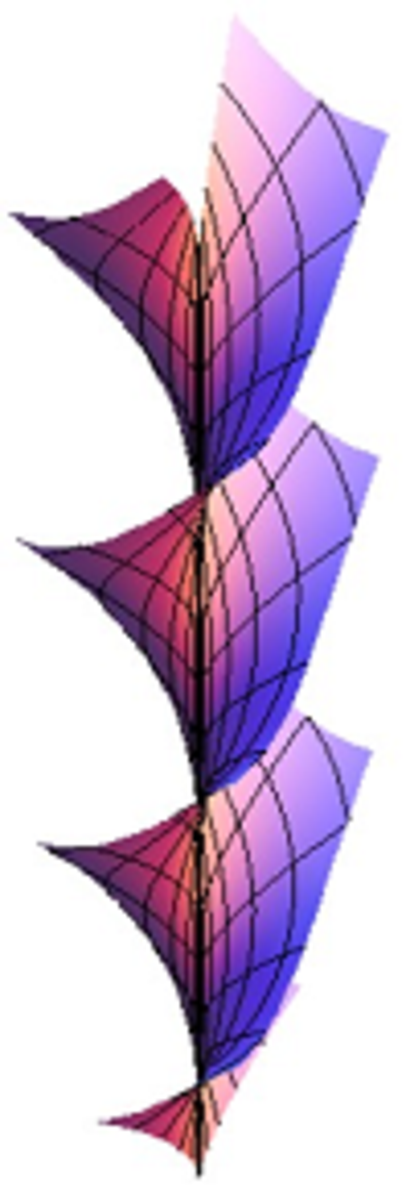}\hspace{20mm}
\includegraphics[width=0.17\linewidth]{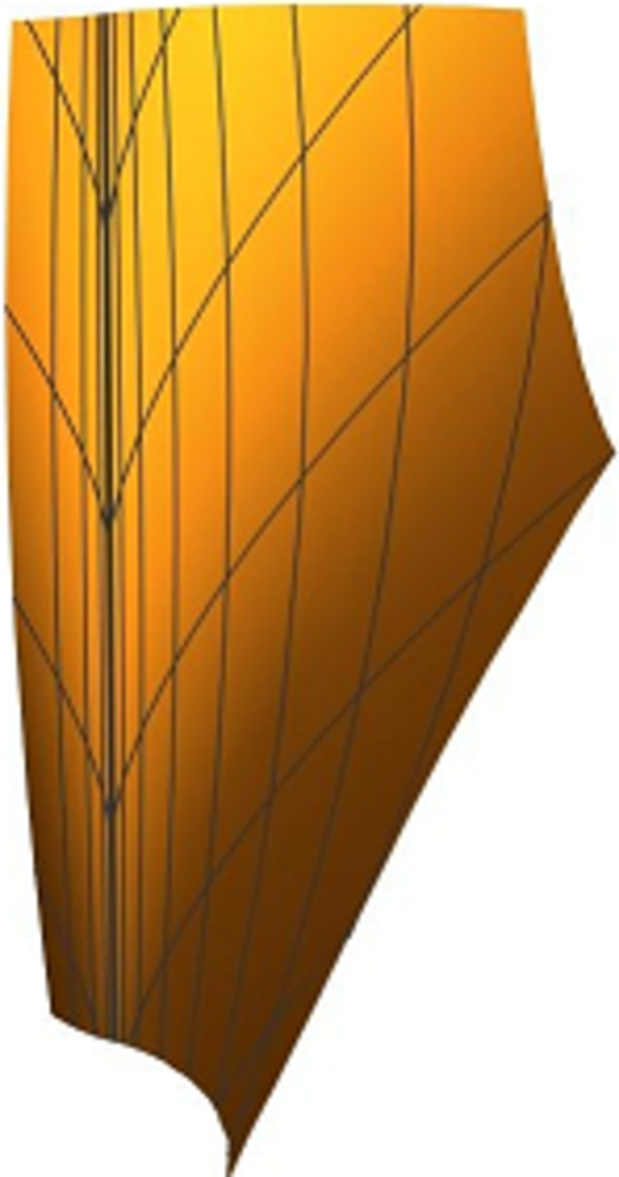}
\label{fig:5-3cusp}
\caption{Helicoidal surface with $5/3$-cuspidal edge 
(the curve $\gamma$, the slice curve $s[r[\gamma]]$, the surface $r$, closer view)}
\end{figure}
\end{example}


\ \\
Nozomi Nakatsuyama,
\\
Muroran Institute of Technology, Muroran 050-8585, Japan,
\\
E-mail address: \href{mailto:23043042@muroran-it.ac.jp}{23043042@muroran-it.ac.jp}
\\
\\
Kentaro Saji,
\\
Faculity of Science, Kobe University, Rokko 1-1, Kobe 657-8501, Japan,
\\
E-mail address: \href{mailto:saji@math.kobe-u.ac.jp}{saji@math.kobe-u.ac.jp}
\\
\\
Runa Shimada,
\\
Faculity of Science, Kobe University, Rokko 1-1, Kobe 657-8501, Japan,
\\
E-mail address: \href{mailto:231s010s@stu.kobe-u.ac.jp}{231s010s@stu.kobe-u.ac.jp}
\\
\\
Masatomo Takahashi, 
\\
Muroran Institute of Technology, Muroran 050-8585, Japan,
\\
E-mail address: \href{mailto:masatomo@muroran-it.ac.jp}{masatomo@muroran-it.ac.jp}

\end{document}